 \documentclass[a4paper,11pt]{article}
 \usepackage[plainpages=false]{hyperref}
 \usepackage{amsfonts,amsmath,amssymb,amsthm}
 \usepackage{latexsym,lscape,rawfonts,mathrsfs}

 \usepackage[dvips]{color}
 \usepackage{multicol}

 \usepackage[numbers]{natbib}

 \usepackage[all]{xy}
 \usepackage{eufrak}
 \usepackage{makeidx}
 \usepackage{graphicx,psfrag}

 \usepackage{array,tabularx}

 \usepackage{setspace}

 \usepackage{appendix}

 \newcommand{\A}{\ensuremath{\mathbb{A}}}
 
 \newcommand{\D}[2]{\ensuremath{ \frac{\partial{#1}}{\partial{#2}}}}

 \newcommand{\Q}{\ensuremath{\mathbb{Q}}}
 \newcommand{\R}{\ensuremath{\mathbb{R}}}
 
 \newcommand{\CP}{\ensuremath{\mathbb{CP}}}
 
 \newcommand{\st}{\ensuremath{\sqrt{-1}}}

 \newcommand{\ddb}{\ensuremath{\partial \bar{\partial}}}

 \newcommand{\KRf}{K\"ahler Ricci flow\;}
 \newcommand{\KRfc}{K\"ahler Ricci flow,\;}
 \newcommand{\KRfd}{K\"ahler Ricci flow.\;}
 \newcommand{\KRF}{K\"ahler Ricci Flow\;}

 \DeclareMathOperator{\Vol}{Vol}
 \DeclareMathOperator{\diam}{diam}

 \newcommand{\Blow}[1]{\ensuremath{\mathbb{CP}^2 \# {#1}\overline{\mathbb{CP}}^2}}

 \newcommand{\norm}[2]{{ \ensuremath{\|} #1 \ensuremath{\|}}_{#2}}
 \newcommand{\snorm}[2]{{ \ensuremath{\left |} #1 \ensuremath{\right |}}_{#2}}
 \newcommand{\sconv}{\ensuremath{ \stackrel{C^\infty}{\longrightarrow}}}

 \makeatletter
 \def\ExtendSymbol#1#2#3#4#5{\ext@arrow 0099{\arrowfill@#1#2#3}{#4}{#5}}
 
 \makeatother

 \definecolor{hao}{rgb}{1,0.5,0}
 \definecolor{miao}{cmyk}{0.5,0,0.2,0.2}
 \definecolor{qiao}{gray}{0.96}

 \newcommand{\Poincare}{Poincar\`{e}\;}

 \newtheorem*{clm}{Claim}
 
 \newtheorem{corollary}{Corollary}[section]
 \newtheorem{proposition}{Proposition}[section]
 \newtheorem{lemma}{Lemma}[section]
 
 \newtheorem{theorem}{Theorem}[section]

 \newtheorem{definition}{Definition}[section]

 \newtheorem{remark}{Remark}[section]

 \newtheorem{theoremin}{Theorem}
 \newtheorem{lemmain}{Lemma}
 \newtheorem{corollaryin}{Corollary}

 \title{Ricci flow on Orbifolds}
 \author{Bing Wang}
 \date{}

\begin{document}
 \maketitle

 \begin{abstract}
   In this paper, we study the behavior of Ricci flows on compact
  orbifolds with finite singularities.  We show that Perelman's pseudolocality theorem also holds on
  orbifold Ricci flow. Using this property, we obtain a weak
  compactness theorem of Ricci flows on orbifolds under some natural technical conditions.
  This generalizes the corresponding theorem on
  manifolds.   As an application, we can use \KRf to find new K\"ahler Einstein metrics on
  some orbifold Fano surfaces.  For example, if $Y$ is a cubic surface with only one ordinary double point or
  $Y$ is an orbifold Fano surface with degree 1 and every
  singularity on it is a rational double point of type $\A_k (1 \leq k \leq 6)$,
  then we can find a KE metric of $Y$ by running \KRf.
 \end{abstract}

 \tableofcontents

 \section{Introduction}

 An important object of Ricci flow is to find Einstein metrics on a
 given manifold. In the seminal paper~\cite{Ha82}, Hamilton showed that staring from any
 metric with positive Ricci curvature on $M^3$, the normalized Ricci flow will always
 converge to an Einstein metric at last.  In the set of K\"ahler manifolds,
 \KRf was developed as an important tool in search of KE (K\"ahler Einstein)
 metrics. In~\cite{Cao85}, based on the fundamental estimate of
 Yau(~\cite{Yau78}), Cao showed the long time existence of \KRf and the convergence of
 \KRf when $c_1(M)\leq 0$. If $c_1(M)>0$, $M$ is called Fano manifold. In this case,
 situations are much more delicate.  $M$ may not have KE metric. So we cannot expect the convergence of the \KRf to a KE metric in general.
 If the existence of KE metric is assumed,  Chen and Tian showed that \KRf converges
 exponentially fast toward the KE  metric if the initial metric has positive bisectional
 curvature (cf. \cite{CT1}, \cite{CT2}).   Using his famous $\mu$-functional,  Perelman developed fundamental estimates along \KRf on Fano manifolds.  He also claimed that the \KRf will always converge
 to the KE metric on any KE manifold.  This result was  generalized to
  manifolds with K\"ahler Ricci solitons by Tian and Zhu (\cite{TZ}).
  If the existence of KE metric is not assumed, there are a lot of works toward the convergence of \KRf
  after G. Perelman's fundamental estimates. For example, important progress can be found in
  (listed in alphabetical order)
  ~\cite{CLH},~\cite{CST},~\cite{CW1},~\cite{CW2},~\cite{Hei},~\cite{PSS},~\cite{PSSW1},
  ~\cite{PSSW2},~\cite{Ru},~\cite{RZZ},~\cite{Se1},~\cite{To},~\cite{TZs} and references therein.\\

  Following Tian's original idea of $\alpha_{\nu, k}$-invariant in~\cite{Tian90}
  and~\cite{Tian91}, Chen and the author (c.f.~\cite{CW3} and~\cite{CW4}) proved
  that the \KRf converges to a KE metric if the $\alpha_{\nu, 1}(M)$ or $\alpha_{\nu, 2}(M)$ is big enough and the flow is tamed.
  They also showed that every 2-dimensional \KRf is tamed.
  Using the calculation of $\alpha_{\nu, 1}$ and $\alpha_{\nu, 2}$ of every Fano surface (c.f.~\cite{ChS},~\cite{SYl}), they showed the convergence of \KRf to a KE metric on every Fano surface $M$
  satisfying $1 \leq c_1^2(M) \leq 6$. This gives a
  flow proof of Calabi's conjecture on Fano surfaces.  The existence of KE metrics on such manifolds were originally proved by
  Tian in~\cite{Tian90}.\\

  A natural question is: can we generalize these results to Fano orbifolds and use \KRf to search the KE metrics on Fano orbifolds? In this paper, we answer this question affirmatively. We use \KRf as a tool to find new KE metrics on some orbifold Fano surfaces.  However, before we can use the orbifold \KRfc  we firstly need some general results of orbifold Ricci flows.
  So we generalize Perelman's Ricci flow theory to orbifold case.  The study of orbifold Ricci flow is pioneered by the work~\cite{CWL},~\cite{Wu},~\cite{Lu1}.\\

  We have the following theorems.

 \begin{theoremin}
   Suppose $Y$ is a Fano orbifold,  $\{(Y, g(t)), 0 \leq t < \infty \}$ is a \KRf solution
   tamed by $\nu$. Then this flow converges to a KE metric if one of the
   following conditions is satisfied.
  \begin{itemize}
  \item $\alpha_{\nu, 1}>\frac{n}{n+1}$.
  \item $\alpha_{\nu, 2}>\frac{n}{n+1}, \; \alpha_{\nu, 1} >
  \frac{1}{2-\frac{(n-1)}{(n+1)\alpha_{\nu, 2}}}$.
  \end{itemize}
 \label{theoremin: tamedconvergence}
 \end{theoremin}

  The tamedness condition originates from Tian's work in~\cite{Tian90}
  (c.f. eqation (0.3) of~\cite{Tian90}). Under \KRfc  it's first defined
  in~\cite{CW4}.   A flow is called tamed by constant $\nu$ if the function
  \begin{align}
      F(\nu, x, t) \triangleq \frac{1}{\nu} \log
      \sum_{\beta=0}^{N_{\nu}} \snorm{S_{\nu, \beta}^t}{h^{\nu}}^2(x)
  \label{eqn: tamed}
  \end{align}
  is uniformly bounded on $Y \times [0, \infty)$.
   Here $\{S_{\nu, \beta}^t\}_{\beta=0}^{N_{\nu}}$
  are orthonormal  basis of $H^0(K_Y^{-\nu})$,
  i.e.,
  \begin{align*}
     \int_Y \langle S_{\nu, \alpha}^{t}, S_{\nu, \beta}^{t}\rangle_{h^{\nu}}
     \omega_t^n=\delta_{\alpha \beta},  \quad 0 \leq \alpha, \beta \leq N_{\nu}=\dim H^0(K_Y^{-\nu})-1;
     \quad
     h=\det g_{i\bar{j}}(t).
  \end{align*}
  Therefore, this theorem gives us a way to search KE metric by
  \KRfd
  $\alpha_{\nu, k}$ are defined as (c.f. Definition~\ref{definition: nualpha}
  for more details)
  \begin{align*}
    \alpha_{\nu, k} \triangleq
      \sup\{ \alpha |  \sup_{\varphi \in \mathscr{P}_{\nu, k}} \int_Y e^{-2\alpha \varphi} \omega_0^n <
      \infty\}
  \end{align*}
  where $\mathscr{P}_{\nu, k}$ is the collection of all functions of
  the form $\displaystyle \frac{1}{2\nu}\log (\sum_{\beta=0}^{k-1} \norm{\tilde{S}_{\nu,
  \beta}}{h_0^{\nu}}^2)$ for some orthonormal basis $\{\tilde{S}_{\nu, \beta}\}_{\beta=0}^{k-1}$
  of a $k$-dimensional subspace of $H^0(K_Y^{-\nu})$.
  Note that $\alpha_{\nu, k}$ are algebraic invariants which can be calculated explicitly in many
  cases, the most
  important thing now is to show when the tamedness condition is satisfied.

  \begin{theoremin}
   Suppose $Y$ is a Fano surface orbifold, $\{(Y, g(t)), 0 \leq t< \infty \}$
   is a \KRf solution. Then there is a constant $\nu$ such that this
   flow is tamed by $\nu$.
  \label{theoremin: surfacetamed}
  \end{theoremin}

  According to these two theorems and the calculations done
  in~\cite{Kosta} and in~\cite{SYl}, we obtain the existence of K\"ahler Einstein
  metrics on some orbifold Fano surfaces.

  \begin{corollaryin}
  Suppose $Y$ is a cubic surface with only one ordinary double point, or $Y$ is
  a degree $1$ del Pezzo surface having only Du Val singularities of type $\A_k$ for $k \leq 6$. Starting from any
   metric $\omega$ satisfying $[\omega]=2\pi c_1(Y)$,
   the \KRf will converge to a KE metric on $Y$.  In particular, $Y$ admits a KE metric.
  \label{corollaryin: KEexample}
  \end{corollaryin}

  Actually, both Theorem~\ref{theoremin: tamedconvergence}
  and Theorem~\ref{theoremin: surfacetamed} have corresponding versions
  in~\cite{CW3} and~\cite{CW4}. Their proofs are also similar to the
  ones in~\cite{CW3} and~\cite{CW4}.\\

   Theorem~\ref{theoremin: tamedconvergence} follows from the
   partial $C^0$-estimated given by the tamedness condition:
    \begin{align}
     \left|\varphi(t) - \sup_M \varphi(t)
    -\frac{1}{\nu}\log \sum_{\beta=0}^{N_{\nu}}
     \snorm{\lambda_{\beta}(t) \tilde{S}_{\nu, \beta}^t}{h_0^{\nu}}^2 \right| <C,
   \label{eqn: spe}
   \end{align}
  where $\varphi(t)$ is the evolving K\"ahler potential,
  $0 < \lambda_0(t) \leq \lambda_1(t) \leq \cdots \leq \lambda_{N_{\nu}}(t)=1 $
  are $N_{\nu}+1$ positive functions of time $t$,
  $\{\tilde{S}_{\nu, \beta}^t\}_{\beta=0}^{N_{\nu}}$ is an orthonormal
  basis of $H^0(K_M^{-\nu})$ under the fixed metric $g_0$.
  Intuitively, inequality (\ref{eqn: spe}) means that we can control $Osc_{M} \varphi(t)$ by
  $\displaystyle \frac{1}{\nu}\log \sum_{\beta=0}^{N_{\nu}} \snorm{\lambda_{\beta}(t) \tilde{S}_{\nu, \beta}^t}{h_0^{\nu}}^2$
  which only blows up along intersections of pluri-anticanonical divisors.
   Therefore, the estimate of $\varphi(t)$ is more or less translated to the study of the property of
  pluri-anticanonical holomorphic sections, which are described by $\alpha_{\nu, k}$.\\

   Theorem~\ref{theoremin: surfacetamed} can be looked as the
   combination of the following two lemmas.

   \begin{lemmain}
   Suppose $Y$ is a Fano orbifold,   $\{(Y^n, g(t)), 0 \leq t < \infty\}$ is a \KRf
   solution satisfying the following two conditions
   \begin{itemize}
   \item No concentration:  There is a constant $K$ such that
     \begin{align*}
        \Vol_{g(t)}(B_{g(t)}(x, r)) \leq Kr^{2n}
     \end{align*}
     for every $(x, t) \in Y \times [0, \infty), r \in (0,K^{-1}]$.
   \item Weak compactness:  For every sequence $t_i \to \infty$, by
   passing to subsequence, we have
   \begin{align*}
      (Y, g(t_i)) \sconv (\hat{Y}, \hat{g}),
   \end{align*}
   where $(\hat{Y}, \hat{g})$ is a Q-Fano normal variety.
   \end{itemize}
   Then this flow is tamed by some big constant $\nu$.
 \label{lemmain: justtamed}
 \end{lemmain}
 Note that the $Q$-Fano normal variety is a normal variety with a
 very ample line bundle whose restriction on the smooth part is the
 plurianticanonical line bundle. The convergence $\sconv$ is the
 convergence in Cheeger-Gromov topology, i.e., it means that the following two
 properties are satisfied simultaneously:
 \begin{itemize}
 \item $d_{GH}(Y_i, \hat{Y}) \to 0$ where $d_{GH}$ is the
      Gromov-Hausdorff distance among metric spaces.
 \item for every smooth compact set
 $K \subset \hat{Y}$, there are diffeomorphisms $\varphi_i: K \to Y_i$
 such that $Im(\varphi_i)$ is a smooth subset of $Y_i$ and $\varphi_i^*(g_i)$
 converges to $\hat{g}$ smoothly on $K$.
 \end{itemize}

 \begin{lemmain}
  Suppose $Y$ is an orbifold Fano surface, $\{(Y, g(t)), 0 \leq t<\infty\}$
 is a \KRf solution, then this flow satisfies the no concentration
 and weak compactness property mentioned in Lemma~\ref{lemmain: justtamed}.
 Moreover, every limit space $(\hat{Y}, \hat{g})$ is a K\"ahler
 Ricci soliton.
 \label{lemmain: surfacetamed}
 \end{lemmain}

 The proof of Lemma~\ref{lemmain: justtamed} follows directly (c.f. Theorem 3.2 of~\cite{CW4})
 if we have the continuity of plurianticanonical holomorphic
 sections --- orthonormal bases of $H^0(K_Y^{-\nu})$ (under metric $g_i$)
 converge to an orthonormal basis of $H^0(K_{\hat{Y}}^{-\nu})$
 (under metric $\hat{g}$)
 whenever $(Y, g_i)$ converge to $(\hat{Y}, \hat{g})$.  Moreover, every
 orthonormal basis of $H^0(K_{\hat{Y}}^{-\nu})$ is a limit of
 orthonromal bases of $H^0(K_Y^{-\nu})$. This fact is assured by H\"ormand's $L^2$-estimate of
 $\bar{\partial}$-operator and an a priori estimate of
 $\snorm{S}{}$ and $\snorm{\nabla S}{}$, where $S$ is a unit norm
 section of $H^0(K_Y^{-1})$ (c.f. Lemma~\ref{lemma: boundh} for the a priori bounds of sections,
 Theorem 3.1 of~\cite{CW4} for the continuity of holomorphic sections).

 The proof of Lemma~\ref{lemmain: surfacetamed} is essentially based
 on Riemannian geometry. It is a corollary of the following Theorem~\ref{theoremin: centerwcpt}.
 In fact, if we define
  $\mathscr{O}(m, c, \sigma, \kappa, E)$ as the  the moduli space of
  compact orbifold Ricci flow solutions $\{(X^m, g(t)), -1 \leq t \leq 1\}$
  whose normalization constant is bounded by $c$, scalar curvature
  bounded by $\sigma$, volume ratio bounded by $\kappa$ from below,
  energy bounded by $E$ (c.f. Definition~\ref{definition: moduli}),
  then this moduli space have no concentration and weak compactness properties.

  \begin{theoremin}
  $\mathscr{O}(m, c, \sigma, \kappa, E)$ satisfies the following two
  properties.
  \begin{itemize}
  \item No concentration. There is a constant $K$ such that
        \begin{align*}
          \Vol_{g(0)}(B_{g(0)}(x, r)) \leq Kr^m
        \end{align*}
     whenever $r \in (0, K^{-1}]$, $x \in X$,
     $\{(X, g(t)), -1 \leq t \leq 1\} \in \mathscr{O}(m, c, \sigma, \kappa, E)$.
  \item Weak compactness. If $\{ (X_i, x_i, g_i(t)) , -1 \leq t \leq 1\} \in \mathscr{O}(m, c, \sigma, \kappa, E)$
 for every $i$, by passing to subsequence if necessary, we have
 \begin{align*}
    (X_i, x_i, g_i(0)) \sconv (\hat{X}, \hat{x}, \hat{g})
 \end{align*}
 for some $C^0$-orbifold $\hat{X}$ in Cheeger-Gromov sense.
  \end{itemize}
 \label{theoremin: centerwcpt}
 \end{theoremin}
  Actually, according to the fact that scalar
 curvature and $\int_Y |Rm|^2 \omega_t^2$ are
 uniformly bounded (c.f. Proposition~\ref{proposition: perelman})
 along \KRf on orbifold Fano surface, it is clear that
 $\{(Y, g(t+T)), -1 \leq t \leq 1\} \in \mathscr{O}(4, 1, \sigma, \kappa, E)$
 for every $T \geq 1$.  Therefore
 Theorem~\ref{theoremin: centerwcpt} applies.  In order to obtain
 Lemma~\ref{lemmain: surfacetamed}, we need to show that the
 limit space $\hat{Y}$ is a K\"ahler Ricci soliton and every orbifold singularity is a
 $C^{\infty}$-orbfiold point (c.f. Definition~\ref{definition: orbifold}).  The first property is a direct
 application of Perelman functional's monotonicity (c.f.~\cite{Se1}), the second property follows from
 Uhlenbeck's removing singularity method (c.f.~\cite{CS}).\\

 Theorem~\ref{theoremin: centerwcpt} is a generalization of the
 corresponding weak compactness theorem in~\cite{CW3}.
 If we assume Perelman's pseudolocality theorem (Theorem 10.3 of~\cite{Pe1})
 holds in orbifold case, then its proof can be
 almost the same as the corresponding theorems in~\cite{CW3}.
 Therefore, an important technical difficulty of this paper is
 the following pseudolocality theorem.

  \begin{theoremin}
  There exists $\eta=\eta(m, \kappa)>0$ with the following property.

  Suppose $\{(X, g(t)), 0 \leq t \leq r_0^2\}$ is a compact orbifold Ricci flow solution.
  Assume that at $t=0$ we have $|\hat{Rm}|(x) \leq r_0^{-2}$
  in $B(x, r_0)$, and $\Vol B(x, r_0) \geq \kappa r_0^m$. Then the estimate
  $|\hat{Rm}|_{g(t)}(y) \leq (\eta r_0)^{-2}$ holds whenever $0 \leq t \leq (\eta r_0)^2$,
  $d_{g(t)}(y, x) < \eta r_0$.
 \label{theoremin: improvedbound}
 \end{theoremin}
 Note that $|\hat{Rm}|$ is defined as
  \begin{align*}
  |\hat{Rm}|(x)=
  \left\{
  \begin{array}{ll}
  |Rm|(x), & \textrm{if $x$ is a smooth point.}\\
  \infty, & \textrm{if $x$ is a singularity.}
  \end{array}
  \right.
  \end{align*}
 The proof of Theorem~\ref{theoremin: improvedbound} is a
 combination of Perelman's point selecting method and maximal principle.
 Note that the manifold version of
 Theorem~\ref{theoremin: improvedbound}  (Theorem 10.3 of~\cite{Pe1})
 is claimed by Perelman without proof. This first
 written proof is given by Lu Peng in~\cite{Lu2} recently.\\

  With Theorem~\ref{theoremin: improvedbound} in hand, we can prove
  Theorem~\ref{theoremin: centerwcpt} as we did in~\cite{CW3}.
  However, we prefer to give a new proof.  In~\cite{CW3},
  the proof of weak compactness theorem is complicated.  A lot of
  efforts are paid to show the locally connectedness of the limit
  space. In other words, we need to show the limit space is an
  orbifold, not a multifold.  We used bubble tree on space time to
  argue by contradiction. If we are able to construct bubble tree on
  a fixed time slice,  then the argument will be much easier.
  In this paper, we achieve this by observing some stability of
  $\int |Rm|^{\frac{m}{2}}$ in unit geodesic balls. \\

  In short, the new ingredients of this paper are listed as follows.

  \begin{itemize}
  \item We offer a method to find KE metrics on orbifold Fano surfaces.
  \item We give a simplified proof of weak compactness theorem,
  i.e., Theorem~\ref{theoremin: centerwcpt}.
  \item We prove the pseudolocality theorem in orbifold Ricci flow.
  \end{itemize}

  It's interesting to compare the two methods used in
  search of KE metrics: the  continuity method and the flow method.
  Suppose $(M, g, J)$ is a K\"ahler manifold with positive first Chern class
  $c_1$,  $\omega$ is the $(1, 1)$-form
  compatible to $g$ and $J$.  The existence of KE metric under the
  complex structure $J$ is equivalent to the solvability of the
  equation
  \begin{align*}
     \det (g_{i\bar{j}} + \frac{\partial^2 \varphi}{\partial z_i \bar{\partial}
     \overline{z_j}})=e^{-u-\varphi} \det (g_{i\bar{j}}),
     \quad
     g_{i\bar{j}} + \frac{\partial^2 \varphi}{\partial z_i \bar{\partial}
     \overline{z_j}}>0,
  \end{align*}
  where $u$ is a smooth function on $M$ satisfying
  \begin{align*}
     u_{i\bar{j}}=g_{i\bar{j}}-R_{i\bar{j}},
     \quad \frac{1}{V} \int_M (e^{-u}-1) \omega^n=1.
  \end{align*}
  In continuity method, we try to solve a family of equation
  ($0 \leq t \leq 1$):
  \begin{align*}
   \left\{
   \begin{array}{l}
     \det (g_{i\bar{j}} + \frac{\partial^2 \varphi}{\partial z_i \bar{\partial}
     \overline{z_j}})=e^{-u-t\varphi} \det (g_{i\bar{j}}),\\
     g_{i\bar{j}} + \frac{\partial^2 \varphi}{\partial z_i \bar{\partial}
     \overline{z_j}}>0.
   \end{array}
   \right.
  \end{align*}
  In K\"ahler Ricci flow method, we try to show the convergence of the parabolic equation solution:
  \begin{align*}
     \D{\varphi}{t} = \log \frac{\det (g_{i\bar{j}} + \frac{\partial^2 \varphi}{\partial z_i \bar{\partial}
     \overline{z_j}})}{\det (g_{i\bar{j}})}
     + \varphi + u.
  \end{align*}
  In both methods, the existence of KE metric is reduced to set up a
  uniform $C^0$-bound of the K\"ahler potential function $\varphi$.
  If $\alpha(M)>\frac{n}{n+1}$, then $\varphi$ is uniformly bounded
  in either case (c.f.~\cite{Tian87},~\cite{Ru},~\cite{CW2}).
  If $\alpha(M) \leq \frac{n}{n+1}$, we need more
  geometric estimates to show the uniform bound of $\varphi$.  Under
  continuity path, these geometric estimates are
  stated by Tian in~\cite{Tian90} and~\cite{Tian91} (c.f. inequality (0.3) of~\cite{Tian90}
  and inequality (5.2) of~\cite{Tian91}).
  In \KRf case, we used a similar statement and called it as
  tamedness condition(c.f. equation (\ref{eqn: tamed})) for simplicity.
  If the continuity path or \KRf is tamed, then the $\varphi$ is
  uniformly bounded if $\alpha_{\nu, k} (k=1, 2)$ is big enough.
  However, if the complex structure is fixed,
  there are slight difference in obtaining the tamedness condition
  between these two methods. The tamedness condition of a \KRf maybe easier
  to verify under the help of Perelman's functional.
  On a continuity path, the tamedness condition is
  conjectured to be true by Tian(c.f. Inequality (5.2.) of~\cite{Tian91}).

  Let's recall how to find the KE metric on K\"ahler surface
  $(M, J)$ whenever $c_1^2(M)=3$ and $(M, J)$ contains Eckhard point.
  It was first found by Tian in~\cite{Tian90} where he used
  continuity method twice.  Note that on the differential manifold
  $M \sim \Blow{6}$, all the complex structures such that $c_1$ positive
  form a connected $4$-dimensional algebraic variety $\mathscr{J}$. Choose
   $J_0 \in \mathscr{J}$ such that $\alpha_{G}(M, J_0)>\frac23$ for some compact group
   $G \subset Aut(M, J_0)$(e.g. Fermat surface).
  By continuity method, there is a KE metric $g_0$ compatible with $J_0$.
  Now connecting $J_0$ and $J$ by a family of complex structures
  $J_t \in \mathscr{J}, 0 \leq t \leq 1$ such that $J_1=J$.
  Choose $\tilde{g}_t$ be a continuous family of metrics compatible with $J_t$.
  Let $I$ be the
  collection of all $t$ such that there exists a KE metric
  $g_t$ compatible with $J_t$. It's easy to show that $I$ is an open
  subset of $[0, 1]$.  In order to prove $I=[0, 1]$, one only need
  to show the closedness of $I$.  Let
  \begin{align*}
     (g_t)_{i\bar{j}}= (\tilde{g}_t)_{i\bar{j}} + \varphi_{i\bar{j}}.
  \end{align*}
  Then it suffices to show a uniform bound of $Osc_M \varphi(t)$ on $I$.
  Since along this curve of complex structures,
  $\alpha_{\nu, 2}(M, J_t)>\frac23, \alpha_{\nu, 1}(M, J_t) \geq \frac23$
  for every $t\in I \subset [0, 1]$ (c.f.~\cite{SYl},~\cite{ChS}),
  it suffices to show the tamedness condition
  (inequality (0.3) of~\cite{Tian90}) on the set $I$. In fact, this tamedness
  condition is guaranteed by the
  weak compactness theorem of KE metrics on $M$(c.f. Proposition 4.2 of~\cite{Tian90}).

  In \KRf method, we are unable to change complex structure.
  Inspired by the continuity method, we also reduce the boundedness
  of $\varphi$ to the tamedness condition since
   $\alpha_{\nu, 2}(M, J)>\frac23, \alpha_{\nu, 1}(M, J) \geq  \frac23$.
  Now in order to show the tamedness condition, we need a weak
  compactness of time slices of a \KRfd This seems to be more
  difficult since each time slice is only a K\"ahler metric, not a
  KE metric, we therefore lose the regularity property of KE
  metrics.   Luckily,  under the help of Perelman's estimates and
  pseudolocality theorem, we are able to show the weak compactness
  theorem(c.f. Theorem 4.4 of~\cite{CW3}).
  Consequently the tamedness condition of the \KRf on
  $M$ holds, so $\varphi$ is uniformly bounded and this flow converges
  to a KE metric.

  Once the weak compactness of time slices is proved, the
  disadvantage of \KRf becomes an advantage: we can prove the
  tamedness condition without changing complex structure.
  This is not easy to be proved under a continuity path when the
  complex structure is fixed.   Suppose we have a differential
  manifold $M$ whose complex structures with positive $c_1$ form a
  space $\mathscr{J}$ satisfying
  \begin{align*}
   \alpha_{\nu, 1}(M, J) \leq \frac{n}{n+1}, \quad \forall \; J \in
   \mathscr{J}.
  \end{align*}
  Without using symmetry of the initial metric,
  we cannot apply continuity method directly to draw
  conclusion about the existence of KE metrics on $(M, J)$.  However, \KRf can still
  possibly be applied.      For example, if $(Y, J)$ is an Fano orbifold
  surface with degree $1$ and with three rational double points of type $\A_5, \A_2$ and $\A_1$.
  Then $J$ is the unique complex structure on
  $Y$ such that $c_1(Y)>0$ (c.f.~\cite{Zhd},~\cite{YQ}).
  According to the calculations in~\cite{Kosta}, we know
  $\alpha_{\nu, 1}(Y)=\frac23$,
   $\alpha_{\nu, 2}>\frac23$. So we are unable to use continuity method
  directly to conclude the existence of KE metric on $(Y, J)$ because of the absence of tamedness condition.
  However, we do have this condition under \KRf by Theorem~\ref{theoremin: surfacetamed}.
  Therefore,  the \KRf  on $(Y, J)$ must converge to a KE metric. \\

  The organization of this paper is as follows. In section 2, we set
  up notations. In section 3, we go over Perelman's theory on Ricci
  flow on orbifolds and prove the pseudolocality theorem
  (Theorem~\ref{theoremin: improvedbound}).  In
  section 4, we give a simplified version of proof of weak
  compactness theorem (Theorem~\ref{theoremin: centerwcpt}).
  In section 5, we give some improved estimates
  of plurianticanonical line bundles and prove
  Theorem~\ref{theoremin: tamedconvergence} and
  Theorem~\ref{theoremin: surfacetamed}. At last, in section 6,
  we give some examples where our theorems can be applied. In
  particular, we show Corollary~\ref{corollaryin: KEexample}.\\

 \noindent {\bf Acknowledgment}
 The author would like to thank his advisor, Xiuxiong Chen, for
 bringing him into this field and for his constant encouragement.
 The author is very grateful to Gang Tian for
 many insightful and inspiring conversations with him.
 Thanks also go to John Lott, Yuanqi Wang, Fang Yuan for many
 interesting discussions.

 \section{Set up of Notations}

  \begin{definition}
   A $C^{\infty} (C^0)$-orbifold $(\hat{X}^m, \hat{g})$ is a topological
   space which is a smooth manifold with a smooth Riemannian metric
   away from finitely many singular points. At every singular point,
   $\hat{X}$ is locally diffeomorphic to a cone over $S^{m-1} / \Gamma$
   for some finite subgroup $\Gamma \subset SO(m)$. Furthermore, at
   such a singular point, the metric is locally the quotient of a
   smooth (continuous) $\Gamma$-invariant metric on $B^{m}$ under the orbifold
   group $\Gamma$.

     A $C^{\infty}(C^0)$-multifold $(\tilde{X}, \tilde{g})$ is a finite union
     of $C^{\infty}(C^0)$-orbifolds after identifying finite points. In other words,
     $\displaystyle \tilde{X}= \coprod_{i=1}^{N} \hat{X}_i / \sim$ where every
     $\displaystyle \hat{X}_i$ is an orbifold, the relation $\sim$ identifies
     finite points of $\displaystyle \coprod_{i=1}^{N} \hat{X}_i$.

     For simplicity, we say a space is a Riemannian orbifold or orbifold (multifold)
     if it is a $C^{\infty}$-orbifold ($C^{\infty}$-multifold).
 \label{definition: orbifold}
 \end{definition}

 \begin{definition}
 For a compact Riemannian orbifold $X^m$ without boundary, we define its isoperimetric
 constant as
     \begin{align*}
    \mathbf{I}(X)  \triangleq
    \inf_{\Omega} \frac{|\partial \Omega|}{\min\{|\Omega|, |X \backslash \Omega|\}^{\frac{m-1}{m}}}
  \end{align*}
    where $\Omega$ runs over all domains with rectifiable boundaries in $X$.

    For a complete Riemannian orbifold $X^m$ with boundary, we define its isoperimetric
 constant as
  \begin{align*}
    \mathbf{I}(X)  \triangleq
    \inf_{\Omega} \frac{|\partial \Omega|}{|\Omega|^{\frac{m-1}{m}}}
  \end{align*}
    where $\Omega$ runs over all domains with rectifiable boundaries
    in the interior of $X$.
 \end{definition}

 \begin{definition}
   A geodesic ball $B(p, \rho)$ is called $\kappa$-noncollapsed if
  $\displaystyle  \frac{\Vol(B(q, s))}{s^m} > \kappa $
  whenever $B(q, s) \subset B(p, \rho)$.

  A Riemannian orbifold $X^m$ is called $\kappa$-noncollapsed on
  scale $r$ if every geodesic ball $B(p, \rho) \subset X$ is
  $\kappa$-noncollapsed whenever $\rho \leq r$.

   A Riemannian orbifold $X^m$ is called $\kappa$-noncollapsed if it
   is $\kappa$-noncollapsed on every scale $r \leq \diam (X^m)$.
 \end{definition}

  \begin{definition}
   Suppose $(x, t)$ is a point in a Ricci
   flow solution. Then  parabolic balls are defined as
   \begin{align*}
   & P^+(x, t, r , \theta)= \{(y, s)| d_{g(t)}(y, x) \leq r, \; t\leq s
    \leq s+\theta\}.\\
   & P^-(x, t, r , \theta)= \{(y, s)| d_{g(t)}(y, x) \leq r, \; t-\theta \leq s \leq s\}.
   \end{align*}
  Geometric parabolic balls are defined as
   \begin{align*}
   & \tilde{P}^+(x, t, r , \theta)= \{(y, s)| d_{g(s)}(y, x) \leq r, \;t\leq s
    \leq s+\theta\}.\\
   & \tilde{P}^-(x, t, r , \theta)= \{(y, s)| d_{g(s)}(y, x) \leq r, \; t-\theta \leq s \leq s\}.\\
   \end{align*}
  \label{definition: parabolicballs }
  \end{definition}

  \begin{definition}
   Suppose $x$ is a point in the Riemannian orbifold $X$. Then we define
   \begin{align*}
     \snorm{\hat{Rm}}{}= \left\{
   \begin{array}{ll}
   &|Rm|(x), \quad \textrm{if $x$ is a smooth point,}\\
   &\infty, \quad \textrm{if $x$ is a singular point.}
   \end{array}
   \right.
   \end{align*}
  \label{definition: rmhat}
  \end{definition}

 \section{Pseudolocality Theorem}
 \subsection{Perelman's Functional and Reduced Distance}

  Denote $\square= \D{}{t} - \triangle$,
  $\square^*=-\D{}{t} - \triangle + R$.

 In our setting, every orbifold only has finite singularities.  All the concepts in~\cite{Pe1}
 can be reestablished in our orbifold case. For example,
 we can define $W$-functional, reduced distance, reduced volume on orbifold Ricci flow.

 \begin{definition}
   Let $(X,g)$ be a Riemannian orbifold, $\tau>0$ a constant, $f$ a smooth function on $X$.
 Define
\begin{align*}
 W(g, \tau, f) &= \int_X \{ \tau(R+ |\nabla f|^2) +f-n \} (4 \pi
 \tau)^{-\frac{n}{2}} e^{-f}dv,\\
 \mu(g, \tau) &= \inf_{\int_X (4 \pi \tau)^{-\frac{n}{2}}e^{-f}dv=1} W(g, \tau, f).
\end{align*}
 \end{definition}
 Since the Sobolev constant of $X$ exists,
  we know $\mu(g, \tau)>-\infty$ and it is achieved by some smooth function $f$.

 Suppose $\{(X, g(t)), 0 \leq t \leq T\}$ is a Ricci flow solution
 on compact orbifold $X$, $u=(4\pi(T-t))^{-\frac{n}{2}}e^{-f}$
 satisfies $\square^*u=0$. Let
 $v=\{(T-t)(2\triangle f - |\nabla f|^2 +R) +f -m\}u$, then
 \begin{align*}
   \square^* v= -2(T-t)|R_{ij}+f_{ij}-\frac{1}{2(T-t)}g_{ij}|^2u
   \leq 0.
 \end{align*}
 This implies that
 \begin{align*}
     \D{}{t} \int_X\{(T-t)(R+|\nabla f|^2) +f -m\}(4\pi
     \tau)^{-\frac{m}{2}}e^{-f}
  =\D{}{t} \int_X v = \int_X \square^* v \leq 0.
 \end{align*}
 It follows that $\mu(g(t), T-t)$ is nondecreasing along Ricci flow.
 From this monotonicity, we can obtain the no-local-collapsing
 theorem.

 \begin{proposition}
   Suppose $\{(X, g(t)), 0 \leq t <T_0\}$
   is a Ricci flow solution on compact orbifold $X$, then there is a
   constant $\kappa$ such that the following property holds.

   Under metric $g(t)$, if scalar curvature norm $|R| \leq r^{-2}$ in $B(x, r)$ for some
   $r<1$,  then $\Vol(B(x, r)) \geq \kappa r^m$.
 \end{proposition}
 The proof of this proposition is the same as Theorem 4.1
 in~\cite{Pe1} if $R$ is replaced by $|Rm|$.
 See~\cite{KL}, ~\cite{SeT} for the improvement to scalar
 curvature.

 \begin{definition}
  Fix a base point $p \in X$.  Let $\mathcal{C}(p, q, \bar{\tau})$ be the
  collection of all smooth curves $\{\gamma(\tau), 0 \leq \tau \leq \bar{\tau}\}$ satisfying
  $\gamma(0)=p, \gamma(\bar{\tau})=q$.  As in~\cite{Pe1}, we define
 \begin{align*}
   \mathcal{L}(\gamma) &= \int_{0}^{\bar{\tau}} \sqrt{\tau} (R +  |\dot{\gamma}(\tau)|^2)
   d\tau,\\
   L(p, q, \bar{\tau})&= \inf_{\gamma \in
   \mathcal{C}(p,q,\bar{\tau})} \mathcal{L}(\gamma),\\
   l(p, q, \bar{\tau}) &=\frac{L(p, q,
   \bar{\tau})}{2\sqrt{\bar{\tau}}}.
 \end{align*}
 \label{definition: rd}
 \end{definition}

  Like manifold case, $L(p, q, \bar{\tau})$ is achieved by some
  shortest $\mathcal{L}$-geodesic $\gamma$.\\

 Under Ricci flow, since the evolution of distance is controlled by
 Ricci curvature. Definition~\ref{definition: rd} yields the following estimate (c.f.~\cite{Ye}).
 \begin{proposition}
  Suppose $|Ric| \leq Cg$ when $0 \leq \tau \leq \bar{\tau}$ for a
  nonnegative constant $C$. Then
  \begin{align*}
    e^{-2C\tau} \frac{d_{g(0)}^2(p, q)}{4\tau} -\frac{nC}{3}\tau
    \leq l(p, q, \tau) \leq e^{2C\tau} \frac{d_{g(0)}^2(p, q)}{4\tau} +
     \frac{nC}{3}\tau.
  \end{align*}
 \end{proposition}
 Therefore, as $\tau \to 0$, $l(p, q, \tau)$
 behaves like $\frac{d_{g(0)}^2(p, q)}{4\tau}$.

 \begin{proposition}
 Let $u(p, q, \tau)$ be the heat kernel of $\square^*$ on $X \times [0,
 \bar{\tau}]$. As $q \to p,  \; \tau \to 0$, we have
 \begin{align*}
     u(p, q, \tau) \sim (4\pi\tau)^{-\frac{n}{2}} e^{-\frac{d_{g(0)}^2(p,
     q)}{4\tau} + \log |\Gamma_p|}.
 \end{align*}
 \end{proposition}

 In the case the underlying space is a manifold,  this approximation
 can be proved by constructing parametrix for the operator
 $\square^*$(c.f.~\cite{CLN} for detailed proof).
 This construction can be applied to orbifold case easily. See~\cite{DSGW} for the construction of parametrix of heat kernel on
 general orbifold under fixed metrics.  This proposition is the
 combination of the corresponding theorems in~\cite{CLN}
 and~\cite{DSGW}.  The proof method is the same,
 so we omit the proof for simplicity.

 \begin{proposition}
    $\square^* \{(4\pi \tau)^{-\frac{n}{2}} e^{-l}\} \leq 0$.
 \end{proposition}

 \begin{proposition}
 Suppose $h$ is the solution of $\square h=0$, then
 \begin{align*}
 \lim_{t \to 0} \int_X hv \leq -\log |\Gamma|h(p, 0).
 \end{align*}
 \label{proposition: deltalimit}
 \end{proposition}

 \begin{proof}
    Direct calculation shows that
 \begin{align*}
   \D{}{t}\{\int_X hv\}=-\int_X h \square^* v= 2 \tau \int_X \snorm{R_{ij}+f_{,ij}-\frac{g_{ij}}{2\tau}}{}^2
  uh \geq 0.
 \end{align*}
 Therefore, $\displaystyle \lim_{t\to 0^-} \int_X hv$ exists if $\int_X hv$ is
 uniformly bounded as $t \to 0^{-}$. However, we can decompose
 $\int_X hv$ as

 \begin{align*}
   \int_X hv &=\int_X [\tau(2\triangle f -|\nabla f|^2 +R) + f
   -n]uh\\
    &=(4\pi \tau)^{-\frac{n}{2}} \int_X [\tau(2\triangle f -|\nabla f|^2 +R) + f -n]e^{-f}h\\
    &=(4\pi \tau)^{-\frac{n}{2}} \{\int_X [\tau(|\nabla f|^2 +R) + f-n]e^{-f}h -2\tau \int_X h \triangle
    e^{-f}\}\\
    &=\underbrace{\int_X [-2\tau \triangle h + (R\tau -n)h]u}_{I}
       +\underbrace{\int_X \tau |\nabla f|^2 uh}_{II}
       +\underbrace{\int_X fuh}_{III}.
 \end{align*}
 Note that $\int_X u \equiv 1$. Term $I$ is uniformly bounded.
 By the gradient estimate of heat equation, as in~\cite{Ni1}, we have
 \begin{align*}
    \tau \frac{|\nabla u|^2}{u^2} \leq (2+C_1\tau) \{ \log (\frac{B}{u\tau^{\frac{n}{2}}} \int_X u) + C_2 \tau\}
 \end{align*}
 for some constants $C_1, C_2$. Together with $\int_X u \equiv 1$, this implies
 \begin{align*}
  II =\int_X \tau |\nabla f|^2 uh \leq (2+C_1\tau) \{ \int_X ( \log B + f + C_2 \tau)uh\}
     \leq C + 3 \int_X fuh,
 \end{align*}
 where $C$ is a constant depending on $X$ and $h$. It follows that
 \begin{align*}
    \int_X hv \leq C' + 4 \int_X fuh.
 \end{align*}
 In order to show $\int_X hv$ have a uniform upper bound, it suffices to
 show that $III=\int_X fuh$ is uniformly bounded from above.

 Around $(p, 0)$, the reduced distance $l$ on $X$ approximates
 $\frac{d^2}{4\tau}$. (See~\cite{Pe1},~\cite{Ye} for more details.)
 As a consequence,
 we have
 \begin{align*}
   \square^* \{(4\pi \tau)^{-\frac{n}{2}} e^{-l(y, \tau) + \log|\Gamma|}\} \leq 0,
   \quad
   \lim_{\tau \to 0} (4\pi \tau)^{-\frac{n}{2}} e^{-l(y, \tau)+ \log |\Gamma|}=\delta_x(y).
 \end{align*}
 Then maximal principle implies that
 \begin{align}
    f(y, \tau) \leq l(y, \tau) - \log |\Gamma|.
 \label{eqn: fl}
 \end{align}
 for every $y \in X$, $0 <\tau \leq 1$.

%

 Inequality (\ref{eqn: fl}) implies
 \begin{align*}
   \limsup_{\tau \to 0} \int_X fuh &\leq \limsup_{\tau \to 0}\int_X (l-\log|\Gamma|)uh\\
    &=-\log|\Gamma| h(p,0) + \limsup_{\tau \to 0} \int_X
    \frac{d^2}{4\tau}uh\\
    &\leq (\frac{n}{2} - \log|\Gamma|) h(p, 0).
 \end{align*}
 The last step holds since the expansion of $u$ around point
 $(p, 0)$ tells us that
 \begin{align*}
  \limsup_{\tau \to 0}\int_X \frac{d^2}{4\tau} uh
  \leq h(p, 0)
  \{\int_{\R^n / \Gamma} \frac{|z|^2}{4} \cdot (4\pi)^{-\frac{n}{2}}
  \cdot e^{-\frac{|z|^2}{4} + \log|\Gamma|}\}
  =\frac{n}{2}h(p,0).
 \end{align*}

 After the uniform upper bound of $\int_X vh$ is set up, by the
 monotonicity of $\int_X vh$, we see
 $\displaystyle  \lim_{\tau \to 0}  \int_X vh$ exists.  Since $\frac{1}{\tau}$
 is not integrable on $[0,1]$, for every $k$, there are small $\tau$'s such that
 $\displaystyle   \D{}{t}\int_X vh \leq \frac{1}{k \tau}$.
 So we can extract a sequence of $\tau_k \to 0$ such that
 \begin{align*}
  \lim_{k \to \infty} 2\tau_k^2
  \int_X \snorm{R_{ij}+f_{,ij}-\frac{g_{ij}}{2\tau}}{}^2  uh
  = \lim_{k \to \infty}  \tau_k \D{}{t}\int_X vh
  \leq \lim_{k \to \infty} \frac{1}{k}=0.
 \end{align*}
 H\"older inequality and Cauchy-Schwartz inequality implies that
 \begin{align*}
   \lim_{\tau_k \to 0} \tau_k \int_X (R + \triangle f - \frac{n}{2\tau_k})uh
   &\leq
   \lim_{\tau_k \to 0} \tau_k \int_X
   \snorm{R_{ij}+f_{,ij}-\frac{g_{ij}}{2\tau_k}}{}uh\\
   &\leq \lim_{\tau_k \to 0} \{\tau_k^2 \int_X
   \snorm{R_{ij}+f_{,ij}-\frac{g_{ij}}{2\tau_k}}{}^2uh\}^{\frac12}
    \cdot \{\int_X uh\}^{\frac12}\\
   &=0.
 \end{align*}
 Therefore,
 \begin{align*}
   \lim_{\tau \to 0} \int_X vh
 &=\lim_{\tau_k \to 0} \int_X vh\\
 &=\lim_{\tau_k \to 0}  \tau_k \int_X [(R + \triangle f) -
 \frac{n}{2\tau_k}]uh
    -\lim_{\tau_k \to 0} \tau_k \int_X u \triangle h  +  \lim_{\tau_k \to 0}\int_X
    (f-\frac{n}{2})uh\\
 &=\lim_{\tau_k \to 0}\int_X
    (f-\frac{n}{2})uh\\
 &\leq -\log|\Gamma|  h(p, 0).
 \end{align*}
 \end{proof}

 \begin{corollary}
   $v \leq 0$.
 \end{corollary}

 \begin{theorem}
    Suppose $h$ is a nonnegative function, there is a large constant $K$ such that
    $\max \{\square h, -\triangle h \}\leq K$
    whenever $t \in [-K^{-1}, 0]$. Then
 \begin{align*}
 \lim_{t \to 0} \int_X hv \leq -\log |\Gamma|h(p, 0).
 \end{align*}
 \label{theorem: deltalimit}
 \end{theorem}

 \begin{proof}
  The monotonicity of $\int_X v$ tells us that
  \begin{align*}
    \int_X v \geq \int_X v|_{t=K^{-1}} \geq \mu(g(-K^{-1}), K^{-1}).
  \end{align*}
  whenever $t \in [-K^{-1}, 0)$.
  The conditions $\square h \leq K$, $v \leq 0$ imply
 \begin{align*}
   \D{}{t}\{\int_X hv\}=\int_X (v \square h  - h\square^*v) \geq K \int_X v-\int_X h \square^* v
  \geq C +2 \tau \int_X \snorm{R_{ij}+f_{,ij}-\frac{g_{ij}}{2\tau}}{}^2  uh
 \end{align*}
 where $C=K\mu(g(-K^{-1}), K^{-1})$, $\tau=-t$.
 In other words,
 \begin{align*}
   \D{}{t}\{C\tau +\int_X hv\}  \geq 2 \tau \int_X \snorm{R_{ij}+f_{,ij}-\frac{g_{ij}}{2\tau}}{}^2
   uh \geq 0.
 \end{align*}
 By the same argument as in
 Proposition~\ref{proposition: deltalimit}, $C\tau + \int_X hv$ is
 uniformly bounded from above. So the limit
 $\displaystyle \lim_{\tau \to 0} \int_X hv=\lim_{\tau \to 0} C\tau + \int_X hv$ exists. There
 is a sequence $\tau_k \to 0$ such that
 \begin{align*}
 2 \tau_k^2 \int_X \snorm{R_{ij}+f_{,ij}-\frac{g_{ij}}{2\tau}}{}^2
 \longrightarrow 0.
 \end{align*}
 This yields that $\displaystyle \lim_{\tau_k \to 0} \int_X (R+ \triangle f - \frac{n}{2\tau_k})uh
 =0$.
 Note $-\triangle h \leq K$, as in Proposition~\ref{proposition: deltalimit}, we have
 \begin{align*}
   \lim_{\tau \to 0} \int_X vh
 &=\lim_{\tau_k \to 0} \left\{ \tau_k \int_X [(R + \triangle f) -
 \frac{n}{2\tau_k}]uh - \tau_k \int_X u \triangle h
 +\int_X (f-\frac{n}{2})uh \right\}\\
 &\leq -\log|\Gamma|  h(p, 0).
 \end{align*}
 \end{proof}

 \subsection{Proof of Pseudolocality Theorem}

\textbf{In this section, we fix $\alpha= \frac{1}{10^6m}$.}

 \begin{theorem}[\textbf{Pseudolocality theorem}]
  There exist $\delta>0, \epsilon>0$ with the following property.
  Suppose $\{(X, g(t)), 0 \leq t \leq \epsilon^2\}$ is an orbifold Ricci flow
  solution satisfying
 \begin{itemize}
 \item  Isoperimetric constant close to Euclidean one:  $\mathbf{I}(B(x, 1)) \geq (1-\delta) \mathbf{I}(\R^n)$
 \item  Scalar curvature bounded from below:  $R \geq -1$ in $B(x, 1)$.
 \end{itemize}
 under metric $g(0)$.  Then in the geometric parabolic ball
 $\tilde{P}^+(x, 0, \epsilon, \epsilon^2)$,
 every point is smooth and $|Rm| \leq \frac{\alpha}{t} +
 \epsilon^{-2}$.
 \end{theorem}

 \begin{remark}
   The condition $\mathbf{I}(B(x, 1))> (1-\delta)\mathbf{I}(\R^n)$
 implies that there is no orbifold singularity in $B(x, 1)$.
 \end{remark}

 \begin{proof}

   Define $\displaystyle F(x, r) \triangleq \sup_{(y, t) \in \tilde{P}^+(x, 0, r, r^2)}\{ |\hat{Rm}| - \frac{\alpha}{t} -
   r^{-2}\}$ where $|\hat{Rm}|$ is defined in Definition~\ref{definition: rmhat}.  Then the conclusion of the theorem is equivalent to
   $F(x, \epsilon) \leq 0$.

    Suppose this theorem is wrong.
   For every $(\delta, \eta) \in \R^+ \times \R^+$, there are
   orbifold Ricci flow solutions violating the property.   So we can
   take a sequence of positive numbers $(\delta_i, \eta_i) \to (0, 0)$
   and orbifold Ricci flow solutions $\{(X_i, x_i, g_i(t)), 0 \leq t \leq \eta_i^2 \}$
   satisfying the initial conditions but
   $F(x_i, \eta_i)>0$.

   Define $\epsilon_i$ to be the infimum of $r$ such that $F(x_i, r) \geq  0$.
   Since $x_{i}$ is a smooth point, we have $\eta_i>\epsilon_i >0$.  For every point
   $(z, t) \in \tilde{P}^+(x_i, 0, \epsilon_i, \epsilon_i^2)$, we have
   \begin{align}
     |\hat{Rm}|_{g_i(t)}(z) -\frac{\alpha}{t}- \epsilon_i^{-2}
     \leq |\hat{Rm}|_{g_i(t_i)}(y_i) -\frac{\alpha}{t}- \epsilon_i^{-2}=0
   \label{eqn: bschoice}
   \end{align}
   for some point  $(y_i, t_i) \in \tilde{P}^+(x_i, 0, \epsilon_i, \epsilon_i^2)$.

   Let $A_i =\alpha \epsilon_i^{-1}=\frac{1}{10^6m \epsilon_i}$.

  \begin{clm}
    Every point in the geometric parabolic ball
    $\tilde{P}^+(x_i, 0, 4A_i \epsilon_i, \epsilon_i^2)$
    is a smooth point.
  \end{clm}
  For convenience, we omit the subindex $i$.
  Suppose that $(p, s)$ is a singular point in
   $\tilde{P}^+(x,0,4A\epsilon, \epsilon^2)$.
  Let
  \begin{align*}
   \eta(y, t) = \phi (\frac{d_{g(t)}(y, x) +200m \sqrt{t}}{10A\epsilon})
  \end{align*}
  where $\phi$ is a cutoff function satisfying the following
  properties. It takes value one on $(-\infty, 1]$ and decreases to
  zero on $[1, 2]$. Moreover, $-\phi'' \leq 10\phi, (\phi')^2 \leq 10\phi$.
   Recall that
  \begin{align*}
     |\hat{Rm}| \leq  \frac{\alpha}{t} + \epsilon^{-2} \leq
     \frac{1+\alpha}{t} < \frac{2}{t}
  \end{align*}
  in the set $\tilde{P}^+(x,0, \epsilon, \epsilon^2)$. In
  particular, every point in $B_{g(t)(x, \sqrt{\frac{t}{2}})}$ is
  smooth and satisfies $|Rm| < \frac{2}{t}$. This
  curvature estimate implies that (c.f. Lemma 8.3 (a) of~\cite{Pe1}, it also holds in orbifold case.)
  \begin{align*}
  \square d \geq -(m-1)(\frac23 \cdot \frac{2}{t} \cdot \sqrt{\frac{t}{2}} + \sqrt{\frac{2}{t}})
  >-4m t^{-\frac12},
  \end{align*}
  where $d(\cdot)=d_{g(t)}(\cdot, x)$. Therefore, as calculated in~\cite{Pe1},
  we have
  \begin{align*}
  \square \eta = \frac{1}{10A\epsilon}(\square d +100m t^{-\frac12})\phi' - \frac{1}{(10A\epsilon)^2} \phi''
      \leq \frac{10\eta}{(10A\epsilon)^2}.
  \end{align*}
  Let $u$ be the fundamental solution of the backward heat equation
  $\square^* u=0$ and $u=\delta_p$ at point $(p, s)$.
  We can calculate
  \begin{align*}
     \D{}{t} \int_X \eta u  &=  \int_X (u \square \eta - \eta   \square^*u)
   =\int_X u \square \eta
   \leq \frac{1}{10(A\epsilon)^2} \int_X u\eta
   \leq \frac{1}{10(A\epsilon)^2} \int_X u
   =\frac{1}{10(A\epsilon)^2}.
  \end{align*}
  It follows that
  \begin{align*}
   \left.\int_X \eta u \right|_{t=0} \geq \left.\int_X \eta u \right|_{t=s} -\frac{s}{10(A\epsilon)^2}
    \geq  1 - \frac{1}{10A^2}.
  \end{align*}
  Similarly, if we let $\bar{\eta}(y, t) = \phi (\frac{d_{g(t)}(y, x) +200m
  \sqrt{t}}{5A\epsilon})$, we can obtain
  \begin{align*}
   \int_{B(x, 10A\epsilon)} u \geq \left. \int_X \bar{\eta} u \right|_{t=0}
   \geq 1-\frac{10}{(5A)^2}.
  \end{align*}
  It forces that
  \begin{align*}
    \left.\int_{B(x, 20A\epsilon) \backslash B(x, 10A\epsilon)} \eta u \right|_{t=0}
   \leq 1-(1-\frac{10}{(5A)^2}) < A^{-2}.
  \end{align*}
  On the other hand, we have
  \begin{align*}
   \D{}{t} \int_X -\eta v &= \int_X (-v \square \eta + \eta \square^*
   v)\\
   &\leq \int_X -v \square \eta\\
   &\leq \frac{1}{10(A\epsilon)^2} \int_X -\eta v,
  \end{align*}
  where we used the fact $-v \geq 0$ and $\square \eta \leq \frac{\eta}{10(A\epsilon)^2}$.
  This inequality together with Theorem~\ref{theorem: deltalimit} implies
  \begin{align*}
        \left.\int_X -\eta v \right|_{t=0}
    \geq e^{-\frac{s}{10(A\epsilon)^2}} \left.\int_X -\eta v\right|_{t=s}
    \geq \log |\Gamma| \eta(x, s) e^{-\frac{s}{10(A\epsilon)^2}} > \frac12 \log |\Gamma|
    \geq \frac12 \log 2.
  \end{align*}
  Let $\tilde{u}=u\eta$ and $\tilde{f}=f -\log \eta$. At $t=0$, as in~\cite{Pe1},
  we can compute
  \begin{align*}
      \frac12 \log 2 &\leq -\int_X  v \eta
  = \int_X \{(-2\triangle f + |\nabla f|^2 -R) s -f +m\} \eta u\\
  &=\int_X \{ -s|\nabla \tilde{f}|^2 -\tilde{f} +m\} \tilde{u}
    +\int_X \{s(\frac{|\nabla \eta|^2}{\eta} -R\eta) -\eta \log
    \eta\} u\\
  &\leq 10A^{-1} + 100\epsilon^2 +\int_X\{-s|\nabla \tilde{f}|^2 -\tilde{f} -m\} \tilde{u}
  \end{align*}
  After rescaling $s$ to be $\frac12$, we obtain
  \begin{align*}
  \left\{
  \begin{array}{ll}
     &\int_{B(x, \frac{20A\epsilon}{\sqrt{2s}})} \{\frac12 (|\nabla \tilde{f}|^2 + \tilde{f} -m)\}
         < -\frac14 \log 2.\\
     &1-A^{-2} <\int_{B(x, \frac{20A\epsilon}{\sqrt{2s}})} \tilde{u} \leq
     1.
  \end{array}
  \right.
  \end{align*}
  This contradicts to the fact that
  $B(x, \frac{20A\epsilon}{\sqrt{2s}}) \subset B(x, \frac{1}{\sqrt{2s}})$
  has almost Euclidean isoperimetric constant
  (c.f. Proposition 3.1 of~\cite{Ni2} for more details).
  So we finish the proof of the claim.\\

  Now we can do as Perelman did in Claim 1 of the proof of Theorem 10.1 of~\cite{Pe1}.
  We can find a point
  $(\bar{x}, \bar{t})$ such that
  \begin{align*}
      |\hat{Rm}|_{g(t)}(z) \leq 4 |\hat{Rm}|_{g(\bar{t})}(\bar{x})
  \end{align*}
  whenever
  \begin{align*}
    (z, t) \in X_{\alpha}, \; 0< t \leq \bar{t},\;
    d_{g(t)}(z, x) \leq d_{g(\bar{t})}(\bar{x}, x) +
    A|\hat{Rm}|_{g(\bar{t})}(\bar{x})^{-\frac12}
  \end{align*}
  where $X_{\alpha}$ is the set of pairs $(z, t)$ satisfying
   $|\hat{Rm}|_{g(t)}(z) \geq \frac{\alpha}{t}$.
  Moreover, we also have $d_{g(\bar{t})}(\bar{x}, x)<(2A+1)
  \epsilon$.  Therefore, the geometric parabolic ball
  $\tilde{P}^+(\bar{x}, 0, A\epsilon, \bar{t})$
  is strictly contained in the geometric parabolic ball
   $\tilde{P}^+(x, 0, 4A\epsilon,\epsilon^2)$.  Therefore, every point around $(\bar{x}, \bar{t})$
  is smooth. We can replace $|\hat{Rm}|$ by $|Rm|$
  and all the arguments of Perelman's proof in~\cite{Pe1}
  apply directly.  For simplicity, we only sketch the basic steps.

  $P^-(\bar{x}, \bar{t}, \frac{1}{10}AQ^{-\frac12}, \frac12 \alpha Q^{-1})$
  is a parabolic ball satisfying $|Rm| \leq
  4Q=4|Rm|_{g(\bar{t})}(\bar{x})$, every point in it is smooth. Then
  by blowup argument, we can show that there is a time
  $\tilde{t} \in [\bar{t}-\frac12\alpha Q^{-1}, \bar{t}]$, such that
  $\int_{B_{g(\tilde{t})}(\bar{x}, \sqrt{\bar{t}-\tilde{t}})} v < -c_0$
  for some positive constant $c_0$, where $v$ is the auxiliary function related to the fundamental solution
  $u=(4\pi(\bar{t}-t))^{-\frac{n}{2}}e^{-f}$ of conjugate heat
  equation, starting from $\delta$-functions at $(\bar{x}, \bar{t})$.
  Under the help of cutoff functions, we can construct a function
  $\tilde{f}$ satisfying
  \begin{align*}
  \left\{
  \begin{array}{ll}
     &\int_{B(x, \frac{20A\epsilon}{\sqrt{2\bar{t}}})} \{\frac12 (|\nabla \tilde{f}|^2 + \tilde{f} -m)\}
         < -\frac12 c_0.\\
     &1-A^{-2} <\int_{B(x, \frac{20A\epsilon}{\sqrt{2\bar{t}}})} \tilde{u} \leq 1.
  \end{array}
  \right.
  \end{align*}
  under the metric $\frac{1}{2\bar{t}}g(0)$. Since
  $B(x, \frac{20A\epsilon}{\sqrt{2\bar{t}}}) \subset B(x, \frac{1}{\sqrt{2\bar{t}}})$
  has almost Euclidean isoperimetric constant as $\epsilon \to 0, A \to
  \infty$,  we know these inequalities cannot hold simultaneously!
 \end{proof}

 \begin{proposition}
   Let $\{(X, g(t)), 0 \leq t \leq 1\}$, $x, \delta, \epsilon$ be the
   same as in the previous theorem.  If in addition,
   $|Rm|<1$ in the ball $B(x, 1)$ at time $t=0$, then
  \begin{align*}
     |Rm|_{g(t)}(y) < (\alpha\epsilon)^{-2}
  \end{align*}
  whenever $0<t< (\alpha \epsilon)^2,  \; dist_{g(t)}(y, x)< \alpha \epsilon$.
 \end{proposition}

 \begin{proof}
  Suppose not. There is a point $(y_0, t_0)$ satisfying
 \begin{align*}
  |Rm|_{g(t_0)}(y_0) \geq (\alpha \epsilon)^{-2}, \quad
  0<t< (\alpha \epsilon)^2,
  \quad d_{g(t_0)}(y_0, x)< \alpha \epsilon.
 \end{align*}
 Check if $Q=|Rm|_{g(t_0)}(y_0)$ can control  $|Rm|$ of  ``previous and
 outside" points. In other words, check if the following property $\clubsuit$ is
 satisfied.
 \begin{align*}
 \clubsuit: \qquad  |Rm|_{g(t)}(z) \leq 4Q, \quad
  \forall \; 0 \leq t \leq t_0,  \quad
  d_{g(t)}(z, x) \leq d_{g(t_0)}(y_0, x) + Q^{-\frac12}.
 \end{align*}
 If not, there is a point $(z, s)$ such that
 \begin{align*}
  |Rm|_{g(s)}(z) > 4Q, \quad  0 < s \leq t_0,  \quad
  d_{g(s)}(z, x) \leq d_{g(t_0)}(y_0, x) + Q^{-\frac12}.
 \end{align*}
 Then we denote $(z, s)$ as $(y_1, t_1)$ and check if the property $\clubsuit$
 is satisfied at this new base point.
 Now matter how many steps this process are performed, the base
 point $(y_k, t_k)$ satisfies
 \begin{align*}
  0 &< t_k \leq t_0 < (\alpha \epsilon)^2,  \\
  d_{g(t_k)}(y_k, x)
  &< d_{g(t_0)}(y_0, x) + \sum_{l=0}^{k-1}  2^{-l} Q^{-\frac12}
  < d_{g(t_0)}(y_0, x) + 2Q^{-\frac12}
  < 3\alpha \epsilon < \epsilon.
 \end{align*}
 Namely, $(y_k, t_k)$ will never escape the compact set
 \begin{align*}
 \Omega=\{(z, s)|  0 \leq s \leq (\alpha \epsilon)^2,
  \; d_{g(s)}(z, x) < 3\alpha \epsilon \}
 \end{align*}
 which has bounded $|Rm|$. During each step, $|Rm|$ doubles at least.
 Therefore, this process must terminate in
 finite steps and property $\clubsuit$ will finally hold.
 Without loss of generality, we can assume property $\clubsuit$ holds already at
 the point $(y_0, t_0) \in \Omega$. Define
 \begin{align*}
   P= \{(z,s)| 0 \leq s \leq t_0,  \quad  d_{g(s)}(z, y_0) < Q^{-\frac12}\}.
 \end{align*}
 Triangle inequality and property $\clubsuit$ implies $|Rm| \leq 4Q$
 holds in $P$.
 Let $\tilde{g}(t)=4Qg(\frac{t}{4Q})$, we have
\begin{align*}
   P= \{(z,s)| 0 \leq s \leq 4Qt_0,  \quad  d_{\tilde{g}(s)}(z, y_0) < 2\}.
 \end{align*}

 From now on, we do all the calculation under the metric
 $\tilde{g}(t)$.

 Define $\eta(z, t)= 5\phi(d(z, y_0) + 100mt)$
 where $\phi$ is the same cutoff function as before. It equals $1$ on $(-\infty, 0]$ and decreases to zero on $[1, 2]$. It satisfies $-\phi'' \leq 10 \phi, \quad (\phi')^2 \leq 10\phi$.
 In $P$, we calculate
 \begin{align*}
  \snorm{\nabla \eta}{}^2&= 25(\phi')^2 \snorm{\nabla d}{}^2 = 25 (\phi')^2 \leq 250 \phi=50\eta,\\
  \square \eta
   &=5(\square d + 100m) \phi' -5\phi''
  \leq -5\phi'' \leq 10 \eta,\\
  \square \eta^{-4} &=-4 \eta^{-5} \square
  \eta -20 \eta^{-6}\snorm{\nabla \eta}{}^2 \geq -40 \eta^{-4} -1000 \eta^{-5}
  =(-40\eta^2- 1000\eta) (\eta^{-6})\\
  &\geq -6000(\eta^{-4})^{\frac32}.
 \end{align*}
 On the other hand, in $P$, we have
 \begin{align*}
   \square \{|Rm|^2(1-\frac{t}{32\alpha})\} &\leq 16|Rm|^3(1-\frac{t}{32\alpha})
   - \frac{1}{32\alpha} |Rm|^2\\
   &\leq (16- \frac{1+16t}{32\alpha})|Rm|^2 \\
   &\leq (16- \frac{1+16t}{32\alpha})|Rm|^3 \\
   &\leq (16- \frac{1+16t}{32\alpha})\{|Rm|^2(1-\frac{t}{32\alpha})\}^{\frac32}\\
   &\leq -6000\{|Rm|^2(1-\frac{t}{32\alpha})\}^{\frac32}.
 \end{align*}
 In these inequalities, we used the fact that $|Rm| \leq 1$, $16- \frac{1+16t}{32\alpha}<-6000<0$ and $1-\frac{t}{32\alpha}>0$
 in $P$.
 $|Rm| \leq 1$ is guaranteed by the choice of $P$. Recall $\alpha=\frac{1}{10^6m}$, so $16- \frac{1+16t}{32\alpha}<-6000$ is
 obvious.  To prove $1-\frac{t}{32\alpha}>0$, we note that
  $Q=|Rm|_{g(t_0)}(y_0)< \frac{\alpha}{t_0} + \epsilon^{-2}$, so we have
 \begin{align*}
 Qt_0< \alpha + t_0 \epsilon^{-2} < \alpha(1+\alpha), \quad
 1-\frac{t}{32\alpha} \geq 1- \frac{4Qt_0}{32\alpha}> 1 -\frac{1+\alpha}{8}>0.
 \end{align*}
 It follows that
 \begin{align*}
   \square \{|Rm|^2(1-\frac{t}{32\alpha})\}
  <-6000 \{|Rm|^2(1-\frac{t}{32\alpha})\}^{\frac32}
 \end{align*}
 in $P$.
 Therefore, $\eta^{-4}$ is a super solution of $\square F=-6000F^{\frac32}$,  $|Rm|^2(1-\frac{t}{32\alpha})$ is a sub solution of this equation. Moreover,
 \begin{align*}
   &|Rm|^2 (1-\frac{t}{32\alpha}) < \frac{1}{4Q} \leq \frac14 (\alpha \epsilon)^2 <\frac{1}{625} < \eta^{-4}, \quad \textrm{whenever} \;
   t=0.\\
   &|Rm|^2 (1-\frac{t}{32\alpha})< \infty= \eta^{-4},
   \quad \textrm{whenever} \; d(z, y)=2.
 \end{align*}
 Therefore, for every point in $P$, $|Rm|^2(1-\frac{t}{32\alpha})$ is controlled
 by $\eta^{-4}$. In particular,  under metric $\tilde{g}$, at point $(y_0, 4Qt_0)$,
 we have
 \begin{align*}
 |Rm|^2(y_0)(1-\frac{4Qt_0}{32\alpha}) &\leq \eta(y_0, 4Qt_0)^{-4}\\
  &=\{5\phi(400mQt_0)\}^{-4}\\
  &\leq 5^{-4} \{\phi(400m\alpha(1+\alpha))\}^{-4}\\
  &\leq 5^{-4} \{\phi(1)\}^{-4}=\frac{1}{625}
 \end{align*}
 On the other hand, recall that $\alpha= \frac{1}{10^6m}$, we have
 \begin{align*}
  |Rm|^2(y_0)(1-\frac{4Qt_0}{32\alpha})=\frac{1}{16}
  (1-\frac{Qt_0}{8\alpha})> \frac{1}{16} (1-\frac{1+\alpha}{8})> \frac{1}{32}.
 \end{align*}
 It follows that $\frac{1}{32} < \frac{1}{625}$. Contradiction!

 \end{proof}

  As a corollary of this proposition, we can obtain the improved
  Pseudolocality theroem.

 \begin{theorem}[\textbf{Improved Pseudolocality Theorem}]
  There exists $\eta=\eta(m, \kappa)>0$ with the following property.

  Suppose $\{(X, g(t)), 0 \leq t \leq r_0^2\}$ is a compact orbifold Ricci flow solution.
  Assume that at $t=0$ we have $|\hat{Rm}|(x) \leq r_0^{-2}$
  in $B(x, r_0)$, and $\Vol B(x, r_0) \geq \kappa r_0^m$. Then the estimate
  $|\hat{Rm}|_{g(t)}(x) \leq (\eta r_0)^{-2}$ holds whenever $0 \leq t \leq (\eta r_0)^2$,
  $d_{g(t)}(x, x) < \eta r_0$.
 \label{theorem: improvedbound}
 \end{theorem}

 \begin{remark}
   Suppose $c_0 \geq -c$ is a constant, then the ``normalized flow" $\D{g}{t}=-Ric + c_0g$
   is just a parabolic rescaling of the flow $\D{g}{t}=-2Ric$. So
   Theorem~\ref{theorem: improvedbound} also hold for ``normalized" Ricci flow solutions
   $\D{g}{t}=-Ric + c_0g$. However, the constant $\eta$ will also depend on $c$ then.
 \end{remark}

\section{Weak Compactness Theroem}

 \textbf{ In this section, $\kappa, E$  are fixed constants. $\hslash, \xi$
  are small constants depending only on $\kappa$ and $m$ by Definition~\ref{definition: hslash}
  and Definition~\ref{definition: xi}.}

\subsection{Choice of Constants}
 \begin{proposition}[Bando,~\cite{Ban90}]
    There exists a constant $\hslash_a=\hslash_a(m, \kappa)$ such that
    the following property holds.

   If $X$ is a $\kappa$-noncollapsed, Ricci-flat ALE orbifold,
         it has unique singularity and small energy, i.e.,
         $\int_X |Rm|^{\frac{m}{2}}d\mu<\hslash_a$,
        then $X$ is a flat cone.
  \label{proposition: gap}
  \end{proposition}

 \begin{proposition}
 Suppose $B(p, \rho)$ is a smooth, Ricci-flat, $\kappa$-noncollapsed geodesic
  ball and $\partial B(p, \rho) \neq \emptyset$. Then there is a small constant
  $\hslash_b=\hslash_b(m, \kappa)<(\frac{1}{2C_0})^{\frac{m}{2}}$ such that
  \begin{align}
   \sup_{B(p, \frac{\rho}{2})} |\nabla^k Rm| \leq
    \frac{C_k}{\rho^{2+k}}  \{\int_{B(p, \rho)} |Rm|^{\frac{m}{2}} d\mu\}^{\frac{2}{m}}
   \label{eqn: econ}
 \end{align}
 whenever $\int_{B(p, \rho)} |Rm|^{\frac{m}{2}} d\mu < \hslash_b$. Here $C_k$ are constants depending only on the dimension $m$.
 In particular, $B(p, \rho)$ satisfies energy concentration
 property. In other words, if $|Rm|(p) \geq \frac{1}{2 \rho^2}$,  then we have
     \begin{align*}
        \int_{B(p, \frac{\rho}{2})} |Rm|^{\frac{m}{2}} d\mu > \hslash_b.
     \end{align*}
 \label{proposition: econ}
 \end{proposition}

  \begin{definition}
    Let $\hslash \triangleq \min\{\hslash_a, \hslash_b\}$.
  \label{definition: hslash}
  \end{definition}

\begin{proposition}
   There is a small constant $\xi_a(\kappa, m)$ such that
  the following properties hold.
  Suppose that $\{(X, g(t)), 0 \leq t \leq 1\}$ is a Ricci flow
  solution on a compact orbifold $X$ which
  is $\kappa$-noncollapsed.  $\Omega \subset X$  and
  $|Rm|_{g(0)}(x) \leq \xi_a^{-\frac32}$
  for every point $x \in \Omega$. Then we have
  \begin{align*}
    |Rm|_{g_i(t)}(x) \leq  \frac{1}{10000m^2}\xi_a^{-2}, \quad
    \forall \; x \in  \Omega', \; t \in [0, 9 \xi_a^2].
  \end{align*}
  where $\Omega'=\{ y \in \Omega | d_{g(0)}(y, \partial \Omega) > \xi_a^{\frac34}\}$.
  \label{proposition: xia}
  \end{proposition}
  \begin{proof}
   Suppose not. There are sequence of $\{(X_i, g_i(t)), 0 \leq t \leq
   1\}$,  $x_i$, $\Omega'$, $\Omega_i$, and $\xi_i \to 0$  violating the statement.

   Blowup them by scale $\xi_i^{-\frac32}$, let $\tilde{g}_i(t)= g_i(\xi_i^{-\frac32}
   t)$.  We can choose a sequence of points
   $y_i \in \Omega_i', t_i \in [0, 9 \xi_i^{\frac12}]$ satisfying
   \begin{align}
   |Rm|_{\tilde{g}_i(t_i)}(y_i)> \frac{1}{10000m^2} \xi_i^{-\frac12} \longrightarrow
   \infty.
   \label{eqn: rmbig}
   \end{align}
   Note that under metric $\tilde{g}_i(0)$,  $|Rm|\leq 1$ in $B(y_i,1)$,
   so inequality (\ref{eqn: rmbig}) contradicts the improved
   pseudolocality theorem!
  \end{proof}

  \begin{proposition}
   Suppose $X$ is an orbifold which is $\kappa$-noncollapsed on
   scale $1$,  $|Rm| \leq 1$ in the smooth geodesic ball $B(x, 1) \subset
   X$.  Then there is a small constant $\xi_b$ such that
   \begin{align*}
     \frac{\Vol(B(y, r))}{r^m} > \frac78 \omega(m)
   \end{align*}
   whenever $y \in B(x, \frac12)$ and $r < \xi_b^{\frac12}$.
   \label{proposition: xib}
  \end{proposition}

  \begin{definition}
   Define $\xi = \min\{\xi_a, \xi_b, (10000m^2 \frac{E}{\hslash})^{-4}\}$.
  \label{definition: xi}
  \end{definition}

\subsection{Refined Sequences}

  The main theorems of this section are almost the same as that of~\cite{CW3}.

\begin{definition}
   Define
  $\mathscr{O}(m, c, \sigma, \kappa, E)$ as the  the moduli space of
  compact orbifold Ricci flow solutions $\{(X, g(t)), -1 \leq t \leq 1\}$
  satisfying:
  \begin{enumerate}
 \item $\D{}{t}g(t) = -Ric_{g(t)} + c_0 g(t)$ where $c_0 $ is a
 constant satisfying $|c_0| \leq c$.
 \item $\displaystyle \norm{R}{L^{\infty}(X \times [-1, 1])} \leq \sigma$.
 \item  $\displaystyle \frac{\Vol_{g(t)}(B_{g(t)}(x, r))}{r^m} \geq \kappa$
    for all $x \in X, t\in [-1, 1], r \in (0, 1]$.
 \item  $ \{\#Sing(X)\} \cdot \hslash + \int_X |Rm|_{g(t)}^{\frac{m}{2}} d\mu_{g(t)} \leq E$
    for all $t \in [-1, 1]$.
 \end{enumerate}
\label{definition: moduli}
\end{definition}
 Clearly, in order this moduli space be really a generalization of
 the $\mathscr{M}(m, c, \sigma, \kappa, E)$ defined in~\cite{CW3},
 we need $m$ to be an even number.

  We want to show the weak compactness and
  uniform isoperimetric constant bound of
 $\mathscr{O}(m, c, \sigma, \kappa, E)$.
  As in~\cite{CW3}, we use refined sequence as a tool to study
 $\mathscr{O}(m, c, \sigma, \kappa, E)$.     After we obtain the
 weak compactness theorem of refined sequence, the properties of
 $\mathscr{O}(m, c, \sigma, \kappa, E)$ follows from routine blowup
 and bubble tree arguments.   However, we'll give a simpler
 proof of the weak compactness theorem of refined sequences.

 As in~\cite{CW3}, we define Refined sequence.
 \begin{definition}
    Let $\{ (X_i^m, g_i(t)), -1 \leq t \leq 1 \}$ be a sequence of Ricci flows on
    closed orbifolds $X_i^m$.  It is called a
    refined sequence if the following properties  are  satisfied for every $i$.
 \begin{enumerate}
 \item $\displaystyle \D{}{t} g_{i} = -Ric_{g_i} + c_i g_i$ and
      $\displaystyle \lim_{i \to \infty} c_i =0$.
 \item  Scalar curvature norm tends to zero:
   $\displaystyle \lim_{i \to \infty}  \norm{R}{L^{\infty}(X_i \times [-1, 1])}=0.$

 \item  For every radius $r$, there exists $N(r)$ such that
  $(X_i, g_i(t))$ is $\kappa$-noncollapsed on scale $r$ for every $t \in [-1, 1]$
  whenever $i>N(r)$.

 \item Energy uniformly bounded by $E$:
 \begin{align*}
  \{\# (Sing(X_i))\} \cdot \hslash +
  \int_{X_i} |Rm|_{g_i(t)}^{\frac{m}{2}} d\mu_{g_i(t)} \leq E, \qquad \forall \; t
 \in [-1,1].
 \end{align*}

 \end{enumerate}
 \label{definition: refined}
 \end{definition}

 In order to show the weak compactness theorem for every refined
 sequence, we need two auxiliary concepts.

 \begin{definition}
    A refined sequence $\{(X_i, g_i(t)), -1 \leq t \leq 1\}$
  is called an E-refined sequence under constraint $H$ if
  under metric $g_i(t)$, we have
  \begin{align}
 \{\# Sing(B(x_0, Q^{-\frac12}))\} \cdot \hslash + \int_{B(x_0,Q^{-\frac12})} |Rm|^{\frac{m}{2}}
 d\mu \geq \hslash,
 \quad \forall t \in [t_0 , t_0+ \xi^2 Q^{-1}],
 \label{eqn: econ}
 \end{align}
 whenever $(x_0,t_0) \in X_i \times [-\frac12, \frac12]$
 and $Q=|Rm|_{g_i(t_0)}(x_0) \geq H$.
  \end{definition}

  \begin{definition}
  An  E-refined sequence $\{(X_i, g_i(t)), -1 \leq t \leq 1\}$
  under constraint $H$
  is called an EV-refined sequence under constraint $(H, K)$ if under metric $g_i(t)$,
  we have
     \begin{align}
       \frac{\Vol{B(x, r)}}{r^m}  \leq K
    \end{align}
    for every $i$ and $(x, t) \in X_i \times [-\frac14, \frac14]$, $r \in (0, 1]$.
  \end{definition}

 When meaning is clear, we omit the constraint when we mention
 $E$-refined and $EV$-refined sequences.  Clearly, an E-refined sequence is a refined sequence whose
 center-part-solutions  satisfy energy concentration property,
  an EV-refined sequence is an E-refined sequence whose
 center-part-solutions  have bounded volume ratios.
 For convenience, we also call a pointed normalized Ricci flow sequence
    $\{ (X_i, x_i, g_i(t)), -1 \leq t \leq  1  \}$ a (E-, EV-)refined
   sequence if  $\{ (X_i, g_i(t)), -1 \leq t \leq  1 \}$ is a (E-, EV-)refined
   sequence.
   Since volume ratio, energy are scaling invariants,
   blowing up a (E-, EV-)refined sequence at proper points
   generates a new (E-, EV-)refined sequence with smaller constraints.\\

 \begin{remark}
  The definition of refined sequence is the same as in~\cite{CW3}.
  However, the definiton of $E-, EV-$refined sequence here is a slight
  different. This is for the
  convenience of a simplified proof of the weak compactness theorem
  of refined sequences.
 \end{remark}

   We first prove the weak compactness of EV-refined sequence.

   \begin{proposition}[\textbf{$C^{1, \frac12}$-Weak Compactness of EV-refined Sequence}]
   Suppose that $\{(X_i, x_i, g_i(t)), -1 \leq t \leq 1\}$
    is an EV-refined sequence, we have
   \begin{align*}
         (X_i, x_i, g_i(0)) \stackrel{C^{1, \frac12}}{\longrightarrow}
          (X_{\infty}, x_{\infty}, g_{\infty})
   \end{align*}
   where $X_{\infty}$ is a Ricci flat ALE orbifold.
   \label{proposition: EVweakc1h}
   \end{proposition}

   \begin{proof}
      As volume ratio upper bound and energy concentration holds, it
      is not hard to see that
    \begin{align*}
        (X_i, x_i, g_i(0)) \overset{C^{1, \frac12}}{\longrightarrow} (X_{\infty}, x_{\infty}, g_{\infty})
    \end{align*}
    for some limit metric space $X_{\infty}$, which has finite singularity and it's
    regular part is a  $C^{1, \frac12}$ manifold. Moreover, by the
    improved  pseudolocality theorem and the almost scalar flat property of the limit sequence,
    every smooth open set of $X_{\infty}$ is  isometric to an open set of a time slice of a scalar flat, hence
    Ricci flat Ricci flow solution. In other words, every open set of the smooth part of $X_{\infty}$
    is Ricci flat. It's not hard to see that
    \begin{align*}
       \sharp Sing(X_{\infty}) \cdot \hslash + \int_{X_{\infty}}
       |Rm|^{\frac{m}{2}} d\mu \leq E.
    \end{align*}
    This energy bound forces that the tangent space of every singular point
    to be a flat cone, but maybe with more than one ends. Also, the
    tangent cone at infinity is a flat cone(c.f.~\cite{BKN}, ~\cite{An90}, ~\cite{Tian90}).   In other
    words, $X_{\infty}$ is a Ricci flat, smooth, ALE multifold with finite energy.
    We need to show that this limit is an orbifold, i.e., for every $p \in X_{\infty}$,
    the tangent space of $p$ is a flat cone with a unique end. This can be done through
    the following two steps.

   \textit{Step1.  \quad Every singular point of $X_{\infty}$ cannot sit on a
   smooth component.  In other words, suppose $p$ is a singular point of $X_{\infty}$,
   then there exists $\delta_0$ depending on
   $p$ such that every  component of  $\partial B(p, \delta) $ has nontrivial $\pi_1$ whenever
    $\delta < \delta_0$.}

    If this statement is wrong, we can choose $\delta_i \to 0$ such
    that
    \begin{align*}
       |\partial E_{\delta_i}| > \frac{7}{8} m\omega(m) \delta_i^{m-1}
    \end{align*}
    where $E_{\delta_i}$ is a component of $\partial B(p, \delta_i)$.
    By taking subsequence if necessary, we can choose
    $X_i \ni p_i \to  p$ satisfying
    \begin{align*}
       |\partial E_{\delta_i}^i| > \frac78 m\omega(m) \delta_i^{m-1}
    \end{align*}
    where $E_{\delta_i}^i$ is some component of $\partial B(p_i, \delta_i)$.
    Moreover, we can let $p_i$ be the point with
    largest Riemannian curvature in $B(p_i, \rho)$
    for some fixed small number $\rho$.
    Define
    \begin{itemize}
    \item
    $ r_i \triangleq \sup \{r | r< \delta_i, \quad \textrm{the largest component of}\; \partial B(p_i, r) \; \textrm{has area ratio}
     \leq \frac78 m \omega(m)\}$
    \item $r'_i \triangleq \inf \{r |\textrm{the ball} \;
    B(p_i, r) \; \textrm{has volume ratio} \leq \frac34 \omega(m) \}$.
    \end{itemize}
   We claim that $r_i' \leq C Q_i^{-\frac12}$ where $Q_i = |Rm|(p_i)$ and $C$ is
   a uniform constant.  Otherwise, by rescaling $Q_i$ to be $1$ and fixing the central time slice
   to be time $0$,  we can take limit for a new $EV$-refined sequence
   \begin{align*}
     \{(X_i, p_i, \tilde{g}_i(t)), 0 \leq t \leq \xi^2 \} \overset{C^{1, \frac12}}{\to}
     \{(\tilde{X}_{\infty}, p_{\infty}, \tilde{g}_{\infty}(t)), 0 \leq t \leq
     \xi^2\},
   \end{align*}
   where the limit is a stable (Ricci flat ) Ricci flow solution on
   a complete manifold $\tilde{X}_{\infty}$. Moreover, the convergence is
   smooth when $t>0$. Therefore, $(\tilde{X}_{\infty}, p_{\infty},  \tilde{g}_{\infty}(0))$ is
   isometric to $(\tilde{X}_{\infty}, p_{\infty},
   \tilde{g}_{\infty}(\xi^2))$.  This forces that
   $(\tilde{X}_{\infty}, p_{\infty}, \tilde{g}_{\infty}(\xi^2))$
   satisfies
   \begin{align*}
    \hslash \leq  \int_{X_{\infty}} |Rm|^{\frac{m}{2}} d\mu \leq E,
    \quad
    \lim_{r \to \infty}  \frac{\Vol(B(p_{\infty}, r))}{r^m} \geq
    \frac34 \omega(m).
   \end{align*}
   simultaneously.  This is impossible!
   Therefore, $r_i' \leq C Q_i^{-\frac12} \to 0$.
   This estimate of $r_i'$ implies that $r_i$ is
   well defined. Moreover, similar blowup argument shows that
   $\displaystyle \lim_{i \to \infty} \frac{r_i}{r_i'}=\infty$.

     Clearly,  $r_i < \delta_i \to 0$.     Rescale $r_i$ to be $1$ to obtain
   a new EV-refined sequence
    \begin{align*}
    \{ (X_i^{(1)}, x_i^{(1)}, g_i^{(1)}(t)), -1 \leq t \leq 1\}
    \end{align*}
    where $x_i^{(1)}=p_i$.  We have convergence
    \begin{align*}
       (X_i^{(1)}, x_i^{(1)}, g_i^{(1)}(0)) \stackrel{C^{1, \frac12}}{\longrightarrow}
          (X_{\infty}^{(1)}, x_{\infty}^{(1)}, g_{\infty}^{(1)}).
    \end{align*}
    By our choice of $r_i$, for every $r>1$, there is a component of
    $\partial B(x_{\infty}^{(1)}, r)$ whose area is at least $\frac78 m\omega(m)r^{m-1}$.
    Therefore,  the ALE space $X_{\infty}^{(1)}$
    has an end whose volume growth is greater than $\frac78 \omega(m)  r^m > \frac12 \omega(m)
    r^m$.  Detach $X_{\infty}$ as union of orbifolds. One of them
    must be ALE space whose volume growth at infinity is exactly $\omega(m)  r^m$,
    the Ricci flatness forces that this ALE component is isometric to
    Euclidean space.     Since one component of $\partial B(x_{\infty}^{1}, 1)$
    has volume $\frac78 m\omega(m)$, $X_{\infty}^{(1)}$ itself
    cannot be Euclidean. Therefore, $X_{\infty}^{(1)}$ must contain
    a singular point which connects a Euclidean  space.   In other
    words,  $X_{\infty}^{(1)}$ contains a singular point $q$ which sit in a smooth
    component.

      If $X_{\infty}^{(1)}$ has more than one singularity,
    we can blowup at point $q$ as before and obtain a
    new bubble $X_{\infty}^{(2)}$. However, a fixed amount of energy
    (at least $\hslash$) will be lost during this process. Therefore such
    process must stop in finite times. Without loss of generality,
    we can assume that $X_{\infty}^{(1)}$ has a unique singularity $q$.
    By the choice of $p_i$, $\displaystyle x_{\infty}^{(1)} = \lim_{i \to \infty} p_i$
    must be singular if
    $X_{\infty}^{(1)}$ contains a singular point.  It follows that
    $X_{\infty}^{(1)}$ has a unique singularity $x_{\infty}^{(1)}$.
    From the previous argument, we already know that $x_{\infty}^{(1)}=q$
    connects a
    Euclidean space.  Since $x_{\infty}^{(1)}$ is the unique singularity, every
    geodesic $\gamma$ connecting $x_{\infty}^{(1)}$ and some point $x$ in the
    Euclidean space must stay in that Euclidean space. Therefore,
    $\partial B(x_{\infty}^{(1)}, 1)$ has a component which is a standard sphere
    whose area is $m\omega(m)> \frac78 m\omega(m)$. So for large
    $i$,  the largest component of $\partial B(p_i, r_i)$
    has area strictly greater than $\frac78 m\omega(m)r_i^{m-1}$.
    This contradicts to the choice of $r_i$!

    \textit{Step 2. Every singular point of $X_{\infty}$ has only one end. In other words, suppose $p$ is a singular
    point of $X_{\infty}$, then there exists $\delta_0$ depending on $p$ such that $\partial B(p, \delta)$
    is connected whenever $\delta < \delta_0$.}

    Suppose not, there is a small $\delta$ such that
     $\partial B(p, \delta)$ is not connected.  Choose $x, y$
     in two different components of $\partial B(p, \delta)$. Let $\gamma$
     be the shortest geodesic connecting $x$ and $y$.  It must pass
     through  $p$.  Suppose  $x_i, y_i, p_i \in X_i$, $\gamma_i \subset X_i$
     satisfy
     \begin{align*}
      x_i \to x, \quad y_i \to y, \quad p_i \to p, \quad \gamma_i
      \to \gamma
     \end{align*}
     where $\gamma_i$ is the shortest geodesic connecting $x_i, y_i$.

     For every $z \in X_i$, we can define
     \begin{align*}
      \mathscr{R} (z)  \triangleq
      \sup \{r | (\# Sing(B(z, r)))\cdot \hslash + \int_{B(z,r)} |Rm|^{\frac{m}{2}} d\mu=\frac12 \hslash.\}
     \end{align*}
     under the metric $g_i(0)$. Clearly, $\mathscr{R}(z)=0$ iff $z$
     is singular.   On $\gamma_i$, let $q_i$ be the point with the smallest
     $\mathscr{R}$ value and define $r_i=\mathscr{R}(q_i)$.
     Note that on orbifold $X_i$, every shortest geodesic connecting two
     smooth points never pass through orbifold singularity.
     This implies that $r_i= \mathscr{R}(q_i)>0$.
       Clearly, $r_i \to 0$. Now, we
     rescale $r_i$ to be $1$ to obtain new EV-refined sequence
     $\{(X_i, q_i, g_i^{(1)}(t)), -1 \leq t \leq 1\}$
     and take limit
     \begin{align*}
           (X_i, q_i, g_i^{(1)}(0)) \stackrel{C^{1, \frac12}}{\longrightarrow}
         (X_{\infty}^{(1)}, q_{\infty}, g_{\infty}^{(1)}).
     \end{align*}
     Clearly, $X_{\infty}^{(1)}$ contains a straight line passing
     through $q_{\infty}$ which we denote as $\gamma_{\infty}$.
     After rescaling, every unit geodesic ball centered on a point of
     $\gamma_{i}$ contains energy not more than $\frac12 \hslash$.  The
     energy concentration property forces that $|Rm|$ is uniformly
     bounded around $\gamma_i$.  So $\gamma_{\infty}$ is a straigt line
     free of singular point.  Detach $X_{\infty}^{(1)}$ as union of
     orbifolds. Then $\gamma_{\infty}$ must stay in one orbifold
     component. Therefore, there is an orbifold component containing
     a straight line. Then the splitting theorem for Ricci flat
     orbifolds applies and forces that component must be $\R \times
     N^{n-1}$.  The ALE condition forces that component must be
     Euclidean space.   Since $X_{\infty}^{(1)}$ contains a Euclidean
     component. From Step 1, we know every singularity cannot stay
     on smooth component. Therefore, $X_{\infty}^{(1)}$ itself must
     be the Euclidean space. So we actually have convergence
     \begin{align}
           \{(X_i, q_i, g_i^{(1)}(t)), 0 < t \leq \xi^2 \} \stackrel{C^{\infty}}{\longrightarrow}
         \{(X_{\infty}^{(1)}, q_{\infty}, g_{\infty}^{(1)}(t)), 0 < t \leq \xi^2 \}.
     \label{eqn: sconv}
     \end{align}
     This forces that $(X_{\infty}^{(1)}, q_{\infty}, g_{\infty}^{(1)}(t))$ is
     Euclidean space for every $t \in [0, \xi^2]$.

     Now return to the choice of $q_i$
     \begin{align*}
       \{\# Sing(B(q_i, r_i))\} \cdot \hslash
       + \int_{B(q_i, r_i)} |Rm|^{\frac{m}{2}} d\mu  =\frac12 \hslash
     \end{align*}
     actually reads $\int_{B(q_i, r_i)} |Rm|^{\frac{m}{2}} d\mu  = \frac12 \hslash$.
     There is a point $q'_i \in B(q_i, r_i)$ satisfying
     \begin{align*}
      Q_i' \triangleq |Rm|(q'_i) > \left(\frac{\hslash}{2\Vol(B(q_i, 2r_i))} \right)^{\frac{2}{m}}
       > \left(\frac{\hslash}{4} \cdot
       \frac{1}{\omega(m)(2r_i)^m } \right)^{\frac{2}{m}}
       > (\frac{\hslash}{4^m \omega(m)})^{\frac{2}{m}} r_i^{-2}
       \to \infty.
     \end{align*}
     On the other hand, the no-singularity-property of
     $X_{\infty}^{(1)}$ implies $Q_i' < Cr_i^{-2}$ for some
     uniform constant $C$.  Therefore, we have
     \begin{align}
         \delta^2 \triangleq (\frac{\hslash}{4^m \omega(m)})^{\frac{2}{m}} <Q_i'
         r_i^2<C.
     \label{eqn: twosidesq'}
     \end{align}
    In particular, $Q_i' \to \infty$ and therefore the energy concentration property
      applies on $q_i'$:
      \begin{align*}
      \int_{B_{g_i(t)}(q_i', (Q_i')^{-\frac12})} |Rm|^{\frac{m}{2}} d\mu \geq
      \hslash, \quad \; \forall \; 0 \leq t \leq \xi^2 (Q_i')^{-1}.
     \end{align*}
     Combining this with inequality (\ref{eqn: twosidesq'}) implies
    \begin{align*}
      \int_{B_{g_i(t)}(q_i', \delta^{-1}r_i)} |Rm|^{\frac{m}{2}} d\mu \geq
      \hslash, \quad \; \forall \; 0 \leq t \leq \frac{\xi^2}{C}r_i^2.
    \end{align*}
    After rescaling, we have
    \begin{align*}
       \int_{B_{g_i^{(1)}(\bar{t})}(q_i', \delta^{-1})} |Rm|^{\frac{m}{2}}
       d\mu \geq \hslash, \quad
       \bar{t}=\frac{\xi^2}{C}.
    \end{align*}
   The smooth convergence (\ref{eqn: sconv}) implies
   that the energy  of $(X_{\infty}, g_{\infty}^{(1)}(\bar{t}))$ is not less than
    $\hslash$. This
    contradicts to the property that $(X_{\infty}, g_{\infty}^{(1)}(\bar{t}))$
    is a Euclidean space! \\

    Therefore, every singular point of $X_{\infty}$ has a unique
    nontrivial end, i.e., it has tangent space $\R^{n} / \Gamma$
    for some nontrivial $\Gamma$.  So $X_{\infty}$ is an orbifold.
  \end{proof}

  \begin{proposition}
   Every $E$-refined sequence is an $EV$-refined sequence.
  \label{proposition: eev}
  \end{proposition}

The next thing we need to do is to improve the convergence
    topology from $C^{1, \frac12}$ to $C^{\infty}$. In light of Shi's estimate, the following
    backward pseudolocality property assures this improvement.

  \begin{proposition}
      Suppose $\{(X_i, x_i, g_i(t)), -1 \leq t \leq 1\}$ is an E-refined
      sequence satisfying $|Rm|_{g_i(0)} \leq r^{-2}$ in $B_{g_i(0)}(x_i, r)$ for
      some $r \in (0, 1)$.

      Then there is a uniform constant $C$ depending on this sequence such that
   \begin{align*}
     |Rm|_{g(t)}(y) \leq C, \; \textrm{whenever} \;
     (y, t) \in P^-(x_i, 0, \frac12 r, 9\xi^2 r^2).
   \end{align*}
   \end{proposition}

     \begin{proof}
   Without loss of generality, we can assume $r=1$.

    Suppose this statement is wrong, there are points $(y_i, t_i) \in P^-(x_i, 0, \frac12, 9 \xi^2)$
    satisfying $|Rm|_{g_i(t_i)}(y_i) \to \infty$. According to  Proposition~\ref{proposition: EVweakc1h}
    and Proposition~\ref{proposition: eev}, we can take limit
    \begin{align*}
    (X_i, y_i, g_i(t_i)) \stackrel{C^{1, \frac12}}{\longrightarrow}
      (Y_{\infty}, y_{\infty}, h_{\infty}),
    \end{align*}
    where $Y_{\infty}$ is a Ricci flat ALE orbifold, $y_{\infty}$ is
    a singular point.


   \begin{clm}
    $B_{g_i(0)}(y_i, \xi^{\frac12}) \subset B_{g_i(t_i)}(y_i, \lambda_m \xi^{\frac12})$
      for large $i$, where $\lambda_m = 1+ \frac{1}{100m}$.
   \end{clm}

         Actually, let $\gamma$ be the shortest geodesic connecting
     $y_i$ and $p \in B_{g_i(0)}(y_i, \xi^{\frac12})$ under metric $g_i(0)$.
     By energy concentration property,  under metric $g_i(t_i)$,
     after deleting (at most) $N=\lfloor \frac{E}{\hslash} \rfloor$
     geodesic balls of radius  $\xi^{\frac34}$, the remainder set  which we denote as $\Omega_i$
     has  uniform Riemannian curvature bounded by $\xi^{-\frac32}$.
     Therefore, according to the choice of $\xi$ (c.f. Proposition~\ref{proposition: xia}),
     we know $|Rm|$ is
     uniformly bounded by $\frac{1}{10000m^2} \xi^{-2}$ on $ \Omega_i' \times [t_i, 0]$
     where
     \begin{align*}
     \Omega_i'= \{x \in \Omega_i| d_{g_i(t_i)}(x, \partial \Omega_i) \geq
     \xi^{\frac34}\},
     \quad [t_i, 0] \subset [t_i, t_i + 9\xi^2].
     \end{align*}
      As the change of length is controlled by integration of Ricci curvature over time, we have
      \begin{align*}
         dist_{g_i(t_i)}(p, y_i) &\leq e^{\frac{1}{10^6m}} length_{g_i(0)}
         \gamma + N \cdot 2\xi^{\frac34}
         \leq (e^{\frac{1}{10^6m}} + N \cdot 2\xi^{\frac14})\xi^{\frac12}
         < \lambda_m \xi^{\frac12},
      \end{align*}
      where the last step follows from the choice of $\xi$.
      The Claim is proved.\\

     Since $y_i \in B_{gi(0)}(x_i, \frac12)$, according to the choice of $\xi$,
    we have
      $\Vol_{g_i(0)}(B_{g_i(0)}(y_i, \xi^{\frac12})) > \frac78 \omega(m) \xi^{\frac{m}{2}}$.
    On the other hand,
    $C^{1, \frac12}$-convergence and volume comparison implies
   \begin{align*}
    \frac{\Vol_{g_i(t_i)}(B_{g_i(t_i)}(y_i, \lambda_m \xi^{\frac12}))}{(\lambda_m \xi^{\frac12})^m}
     \to \frac{\Vol(B(y_{\infty}, \lambda_m \xi^{\frac12}))}{(\lambda_m
    \xi^{\frac12})^m}
    \leq \lim_{r \to 0} \frac{\Vol(B(y_{\infty}, r))}{r^m}
    = \frac{\omega(m)}{|\Gamma(y_{\infty})|}.
   \end{align*}
   As volume change is controlled by integration of scalar curvature
   which is tending to zero, we know
   \begin{align*}
     \lim_{i \to \infty}
    \frac{\Vol_{g_i(0)}(B_{g_i(t_i)}(y_i, \lambda_m \xi^{\frac12}))}{(\lambda_m \xi^{\frac12})^m}
    = \lim_{i \to \infty}
    \frac{\Vol_{g_i(t_i)}(B_{g_i(t_i)}(y_i, \lambda_m \xi^{\frac12}))}{(\lambda_m \xi^{\frac12})^m}
    \leq \frac{\omega(m)}{|\Gamma(p_{\infty})|}
    \leq \frac12 \omega(m).
   \end{align*}
   Therefore, for large $i$, we have
   \begin{align*}
     \frac78 \omega(m) \xi^{\frac{m}{2}}
     <\Vol_{g_i(0)}(B_{g_i(0)}(y_i, \xi^{\frac12}))
     \leq \Vol_{g_i(0)}(B_{g_i(t_i)}(y_i, \lambda_m \xi^{\frac12}))
     <\frac34 \omega(m)(\lambda_m)^m \xi^{\frac{m}{2}}.
   \end{align*}
   It implies $e^{\frac{1}{100}} >(1+\frac{1}{100m})^m=\lambda_m^m > \frac76$ which is impossible!
   This contradiction establish the proof of backward pseudolocality.
   \end{proof}

  Using this backward pseudolocality theorem, we can
  improve $C^{1,\frac12}$-convergence to $C^{\infty}$-convergence.

   \begin{proposition}[\textbf{$C^{\infty}$-Weak Compactness of EV-refined Sequence}]
   Suppose that $\{(X_i, x_i, g_i(t)), -1 \leq t \leq 1\}$
    is an EV-refined sequence, we have
   \begin{align*}
         (X_i, x_i, g_i(0)) \stackrel{C^{\infty}}{\longrightarrow}
          (X_{\infty}, x_{\infty}, g_{\infty})
   \end{align*}
   where $X_{\infty}$ is a Ricci flat ALE orbifold.
   \label{proposition: EVweak}
   \end{proposition}

   \begin{proposition}
   Every refined sequence is an $E$-refined sequence.
   \end{proposition}

    \begin{proof}
     Suppose not.
      Then by delicate selecting of base points and
     blowup, we can find an $E$-refined sequence
     $\{(X_i, x_i, g_i(t)), -1 \leq t \leq 1\}$ under constraint  $2$ satisfying
     $|Rm|_{g_i(0)}(x_i)=1$ and energy concentration fails at
     $(x_i, 0)$, i.e.,
     \begin{align*}
       \sharp(Sing(B_{g_i(t_i)}(x_i, 1))) + \int_{B_{g_i(t_i)}(x_i, 1)} |Rm|^{\frac{m}{2}} d\mu <
        \hslash
     \end{align*}
     for some $t_i \in [0, \xi^2]$. This means that, under metric
     $g_i(t_i)$,
     $B(x_i, 1)$ is free of singularity and
     $\int_{B(x_i, 1)} |Rm|^{\frac{m}{2}} d\mu <\hslash$.
     The energy concentration
     property implies that $|Rm| \leq 4$ in $B(x_i, \frac12)$ under the
     metric $g_i(t_i)$. Therefore, by the backward pseudolocality,
     we have
     \begin{align*}
      |Rm|_{g_i(t)}(x) \leq 4C, \quad \forall \; (x, t) \in P^-(x_i, t_i, \frac14, \frac94 \xi^2).
     \end{align*}
    In particular, $|Rm|_{g_i(0)}(y) \leq 4C$ for every $y \in B_{g_i(t_i)}(x_i, \frac14)$.
    Based at $(x_i, t_i)$, we can take the smooth limit of
    $P^-(x_i, t_i, \frac18, 2\xi^2) \subset P^-(x_i, t_i, \frac14, \frac94 \xi^2)$,
    which will be a Ricci flat Ricci flow solution. Therefore,
    \begin{align}
       \lim_{i \to \infty} |Rm|_{g_i(t_i)}(x_i) = \lim_{i \to
       \infty} |Rm|_{g_i(0)}(x_i)=1.
    \label{eqn: backextend}
    \end{align}
    On the other hand, the sequence $\{(X_i, x_i, g_i(t)), -1 \leq t \leq 1\}$
    is an $EV$-refined sequence by Proposition~\ref{proposition: eev},
    the $C^{\infty}$-weak compactness theorem for $EV$-refined sequence (under constraint $(2, K_0)$) implies
    \begin{align*}
     (X_i, x_i, g_i(t_i)) \overset{C^{\infty}}{\longrightarrow} (X_{\infty}, x_{\infty}, g_{\infty})
    \end{align*}
    for some Ricci flat ALE orbifold $X_{\infty}$.
    Clearly, $B(x_{\infty}, 1)$ is free of singularity. So Moser
    iteration of $|Rm|$ implies that $|Rm|(x_{\infty})<\frac12$.
    It follows that $\displaystyle \lim_{i \to \infty} |Rm|_{g_i(t_i)}(x_i) \leq
    \frac12$ which contradicts to equation (\ref{eqn: backextend})!
   \end{proof}

   It follows directly the following theorem.

   \begin{theorem}[\textbf{Weak compactness of refined sequence}]
   Suppose $\{(X_i, x_i, g_i(t)), -1 \leq t \leq 1\}$
    is a refined sequence. Then  we have
   \begin{align*}
         (X_i, x_i, g_i(0)) \stackrel{C^{\infty}}{\longrightarrow}
          (X_{\infty}, x_{\infty}, g_{\infty})
   \end{align*}
   for some Ricci flat, ALE orbifold $X_{\infty}$.
   \label{theorem: refinedweak}
   \end{theorem}

\subsection{Applications of Refined Sequences}

  After we obtain this smooth weak convergence, we can use refined
  sequence as a tool to study the moduli space
  $\mathscr{O}(m, c, \sigma, \kappa, E)$ which is defined at the
  begining of this section.   Using the same argument as
  in~\cite{CW3}, we can obtain the following theorems.

  \begin{theorem}[\textbf{No Volume Concentration and Weak Compactness}]
  $\mathscr{O}(m, c, \sigma, \kappa, E)$ satisfies the following two
  properties.
  \begin{itemize}
  \item No volume concentration. There is a constant $K$ such that
        \begin{align*}
          \Vol_{g(0)}(B_{g(0)}(x, r)) \leq Kr^m
        \end{align*}
     whenever $r \in (0, K^{-1}]$, $x \in X$,
     $\{(X, g(t)), -1 \leq t \leq 1\} \in \mathscr{O}(m, c, \sigma, \kappa, E)$.
  \item Weak compactness. If $\{ (X_i, x_i, g_i(t)) , -1 \leq t \leq 1\} \in \mathscr{O}(m, c, \sigma, \kappa, E)$
 for every $i$, by passing to subsequence if necessary, we have
 \begin{align*}
    (X_i, x_i, g_i(0)) \sconv (\hat{X}, \hat{x}, \hat{g})
 \end{align*}
 for some $C^0$-orbifold $\hat{X}$ in Cheeger-Gromov sense.
  \end{itemize}
 \label{theorem: centerwcpt}
 \end{theorem}

 \begin{theorem}[\textbf{Isoperimetric Constants}]
  There is a constant $\iota=\iota(m, c, \sigma, \kappa, E, D)$ such
  that the following property holds.

   If $\{(X, g(t)), -1 \leq t \leq 1\} \in \mathscr{O}(m, c, \sigma, \kappa, E)$
 and $\diam_{g(0)}(X)<D$, then
 \begin{align*}
   \mathbf{I}(X, g(0)) > \iota.
 \end{align*}
 \label{theorem: centersob}
 \end{theorem}

 Theorem 4.5 of~\cite{CW3} can be improved as
 $\mathcal{OS}(m, \sigma, \kappa, E, V)$---the moduli space of compact
 gradient shrinking Ricci soliton
 orbifolds---is compact.

 \section{\KRF on Fano Orbifolds}

 \subsection{Some Estimates}
 All the estimates developed under the \KRf on Fano manifolds hold
 for Fano orbifolds. We list the important ones and only give sketch
 of proofs if the statement is not obvious.

 \begin{proposition}[Perelman, c.f.~\cite{SeT}]
 Suppose $\{(Y^n, g(t)), 0 \leq t < \infty \}$ is a \KRf solution on
 Fano orbifold $Y^n$.
 There are two positive constants $\mathcal{B}, \kappa$ depending only on this flow such that the
 following two estimates hold.
 \begin{enumerate}
 \item Under metric $g(t)$,  let $R$ be the scalar curvature,
 $-u$ be the normalized Ricci potential, i.e.,
 \begin{align*}
 Ric-\omega_{\varphi(t)}= - \st \ddb{u}, \quad
 \frac{1}{V} \int_Y e^{-u} \omega_{\varphi(t)}^n=1.
 \end{align*}
 Then we have
 \begin{align*}
      \norm{R}{C^0} + \diam Y +
      \norm{u}{C^0} + \norm{\nabla u}{C^0} < \mathcal{B}.
 \end{align*}
 \item  Under metric $g(t)$,  $ \displaystyle
      \frac{\Vol(B(x, r))}{r^{2n}} > \kappa$ for every
       $r \in (0, 1)$, $(x, t) \in Y \times [0, \infty)$.
 \end{enumerate}
 \label{proposition: perelman}
 \end{proposition}

 \begin{proof}
   When scalar curvature norm $|R|$ is uniformly bounded, the second
   estimate becomes a direct corollary of the general noncollapsing
   theorem.  So we only need to show the first estimate. The proof
   is almost the same as the manifold case.

   First, note that Green's function exists on every compact orbifold,
   and Perelman's functional behaves the same as in manifold case.
   Same as in~\cite{SeT}, we can apply Green's function and Perelman's functional
   to obtain a uniform lower bound of $u(t)$ where $-u(t)$ is the normalized Ricci potential.

   Second, since $u(t)$ is uniformly bounded from below, we can find
   a big constant $B$ such that $u+2B>B$. Then maximal
   principle tells us that there is a constant $C$ such that
   \begin{align*}
     \frac{\triangle u}{u+2B}<C, \quad
     \frac{|\nabla u|^2}{u+2B}<C.
   \end{align*}
   By the second inequality,  we know $u$ is Lipshitz.  Therefore, $u$ will be bounded whenever diameter is
   bounded.

    Third, diameter is bounded.  Suppose that diameter is unbounded,
   we can find a sequence of annulus $A_i=B_{g(t_i)}(x_i, 2^{i+2}) \backslash B_{g(t_i)}(x_i, 2^{i-2})$
   such that the following properties hold.
   \begin{itemize}
   \item  The closure $\overline{A_i}$ contains no singular point.
   \item  $\Vol_{g(t_i)}(A_i) \to 0$.
   \item  Under metric $g(t_i)$, $\frac{\Vol(B(x_i, 2^{i+2}) \backslash B(x_i, 2^{i-2}))}
         {\Vol(B(x_i, 2^{i+1}) \backslash B(x_i, 2^{i-1}))} < 2^{10n}$.
   \end{itemize}
   The reason we can do this is that $Y$ contains only finite
   singularities.    Then by taking a proper cutoff function whose
   support is in $A_i$, we can deduce that Perelman's functional
   $\mu(g_0, \frac12)$ must tend to $-\infty$. Impossible!
 \end{proof}

 The following estimates on orbifolds are exactly the same as the corresponding estimates on
 manifolds.

 \begin{proposition}[\cite{Zhq}, \cite{Ye}]
 $\{(Y^n, g(t)), 0 \leq t <\infty \}$ is a \KRf on Fano orbifold
 $Y^n$. Then there is a uniform Sobolev constant $C_S$ along this
 flow. In other words, for every
 $f \in C^{\infty}(Y)$, we have
 \begin{align*}
       (\int_Y |f|^{\frac{2n}{n-1}}
       \omega_{\varphi}^n)^{\frac{n-1}{n}}
       \leq
       C_S\{\int_Y |\nabla f|^2 \omega_{\varphi}^n
       + \frac{1}{V^{\frac{1}{n}}} \int_Y |f|^2  \omega_{\varphi}^n\}.
 \end{align*}
 \label{proposition: sobolev}
 \end{proposition}

  \begin{proposition}[c.f.~\cite{Fu2},~\cite{TZ}]
  $\{(Y^n, g(t)), 0 \leq t <\infty \}$ is a \KRf on Fano orbifold $Y^n$.
 Then there is a uniform weak Poincar\`e constant $C_P$ along this flow.
 Namely, for every nonnegative function $f \in C^{\infty}(Y)$, we have
   \begin{align*}
        \frac{1}{V} \int_Y f^2 \omega_{\varphi}^n   \leq
          C_P\{\frac{1}{V} \int_Y |\nabla f|^2 \omega_{\varphi}^n
           + (\frac{1}{V} \int_Y f  \omega_{\varphi}^n)^2\}.
   \end{align*}
 \label{proposition: poincare}
 \end{proposition}

 \begin{proposition}[c.f. \cite{PSS}, \cite{CW2}]
    By properly choosing initial condition, we have
 \begin{align*}
 \norm{\dot{\varphi}}{C^0} + \norm{\nabla \dot{\varphi}}{C^0}<C
 \end{align*}
 for some constant $C$ independent of time $t$.
 \label{proposition: dotphi}
 \end{proposition}

 \begin{proposition}[\cite{CW2}]
   There is a constant $C$ such that
 \begin{align}
   \frac{1}{V} \int_Y (-\varphi) \omega_{\varphi}^n \leq
  n \sup_Y \varphi  - \sum_{i=0}^{n-1}
     \frac{i}{V} \int_Y \st \partial \varphi \wedge
     \bar{\partial} \varphi \wedge \omega^i \wedge
     \omega_{\varphi}^{n-1-i} + C.
  \label{eqn: dsupphi}
 \end{align}
 \end{proposition}

 \begin{proposition}[\cite{Ru}, c.f. \cite{CW2}]
 $\{(Y^n, g(t)), 0 \leq t <\infty \}$ is a \KRf on Fano orbifold $Y^n$.
 Then the following conditions are equivalent.
 \begin{itemize}
 \item $\varphi$ is uniformly bounded.
 \item $\displaystyle \sup_Y \varphi$ is uniformly bounded from above.
 \item $\displaystyle \inf_Y \varphi$  is uniformly bounded from below.
 \item $\int_Y \varphi \omega^n$ is uniformly bounded from above.
 \item $\int_Y (-\varphi) \omega_{\varphi}^n$ is
   uniformly bounded from above.
 \item $I_{\omega}(\varphi)$ is uniformly bounded.
 \item $Osc_{Y} \varphi$ is uniformly bounded.
 \end{itemize}
 \label{proposition: conditions}
 \end{proposition}

 \subsection{Tamed Condition by Two Functions: $F$ and $\mathcal{F}$}
  This subsection is similar to the corresponding part
  in~\cite{CW3}. However, we compare different metrics on the line
  bundle to study the tamedness condition.

  Along the \KRF, we have
  \begin{align*}
  \omega_{\varphi(t)}= \omega_0 + \st \ddb \varphi(t), \quad
  \st \ddb \dot{\varphi}(t) = \omega_{\varphi(t)} - Ric_{\omega_{\varphi(t)}}
  \end{align*}
   For simplicity, we omit the subindex $t$. Let $h$ be the metric on $K_Y^{-1}$ induced directly by the
   metric on $Y$, i.e., $h= \det g_{i\bar{j}}$. Let
   $l=e^{-\dot{\varphi}}h$. Clearly,  we have
  \begin{align*}
     -\st \ddb \log \snorm{S}{l}^2+
      \st \ddb \log \snorm{S}{h}^2
    =\st \ddb \dot{\varphi} =\omega_{\varphi} -
    Ric_{\omega_{\varphi}}
  \end{align*}
  It follows that $\st \ddb \log \snorm{S}{l}^2= \omega_{\varphi}$.

  \begin{definition}
   Choose $\{T_{\nu, \beta}^t\}_{\beta=0}^{N_{\nu}}$
   as orthonormal basis of $H^0(K_Y^{-\nu})$ under the metric $h^{\nu}$.
  Then
  \begin{align*}
      &F(\nu, x, t) = \frac{1}{\nu} \log
      \sum_{\beta=0}^{N_{\nu}} \snorm{T_{\nu, \beta}^t}{h^{\nu}}^2(x),\\
      &G(\nu, x, t)=\sum_{\beta=0}^{N_{\nu}} \snorm{\nabla T_{\nu, \beta}^t}{h^{\nu}}^2(x)
  \end{align*}
  are well defined functions on $Y \times [0, \infty)$.

  We call the flow is tamed by $\nu$ if $F(\nu, \cdot, \cdot)$
  is a bounded function on $Y \times [0, \infty)$.
  \end{definition}

  \begin{remark}
   If $Y$ is an orbifold, $K_Y^{-\nu}$ is a line bundle if and only if $\nu$ is an integer multiple of of Gorenstein index
  of $Y$. We call such $\nu$ as appropriate. In this note, we always choose $\nu$ as appropriate ones.
  \end{remark}

  Clearly, $G= \triangle e^{\nu F} - \nu R e^{\nu F}$.
  Fix $(x, t)$, by rotating basis, we can always find a section $T$ such that
  \begin{align*}
     \int_Y \snorm{T}{h^{\nu}(t)}^2 \omega_{\varphi}^n =1, \quad  e^{\nu F(\nu, x, t)} =
     \snorm{T}{h^{\nu}(t)}^2(x).
  \end{align*}
  There also exists a section $T'$ such that
  \begin{align*}
     \int_Y \snorm{T'}{h^{\nu}(t)}^2 \omega_{\varphi}^n =1, \quad  G(\nu, x, t) =
     \snorm{\nabla T'}{h^{\nu}(t)}^2(x).
  \end{align*}

  \begin{definition}
  Choose $\{S_{\nu, \beta}^t\}_{\beta=0}^{N_{\nu}}$
   as orthonormal basis of $H^0(K_Y^{-\nu})$ under the metric $l^{\nu}$.
  Then
  \begin{align*}
      &\mathcal{F}(\nu, x, t) = \frac{1}{\nu} \log
      \sum_{\beta=0}^{N_{\nu}} \snorm{S_{\nu,
      \beta}^t}{l^{\nu}}^2(x),\\
      &\mathcal{G}(\nu, x, t)=\sum_{\beta=0}^{N_{\nu}} \snorm{\nabla S_{\nu, \beta}^t}{l^{\nu}}^2(x).
  \end{align*}
  are well defined functions on $Y \times [0, \infty)$.
  \end{definition}
  Similarly, $\mathcal{G}= \triangle e^{\nu \mathcal{F}} - n\nu e^{\nu \mathcal{F}}$.
  Fix $(x, t)$, by rotating basis, there are unit norm sections $S$
  and $S'$ such that
  \begin{align*}
     &\int_Y \snorm{S}{l^{\nu}(t)}^2 \omega_{\varphi}^n =1, \quad  e^{\nu \mathcal{F}(\nu, x, t)} =
     \snorm{S}{l^{\nu}(t)}^2(x);\\
     &\int_Y \snorm{S'}{l^{\nu}(t)}^2 \omega_{\varphi}^n =1, \quad  \mathcal{G}(\nu, x, t) =
     \snorm{\nabla S'}{l^{\nu}(t)}^2(x).
  \end{align*}

  At point $(x, t)$, we have
  \begin{align*}
   e^{\nu \mathcal{F}}= \snorm{S}{l^{\nu}(t)}^2= e^{-\nu
   \dot{\varphi}}\snorm{S}{h^{\nu}(t)}^2
   = e^{-\nu \dot{\varphi}} \cdot \frac{\snorm{S}{h^{\nu}(t)}^2(x)}{\int_Y \snorm{S}{h^{\nu}(t)}^2 \omega_{\varphi}^n}
    \cdot \int_Y \snorm{S}{h^{\nu}(t)}^2 \omega_{\varphi}^n
   \leq e^{\nu (F-\dot{\varphi} + \snorm{\dot{\varphi}}{C^0})}
   \leq e^{2\nu \mathcal{B}} e^{\nu F}.
  \end{align*}
  Similarly, we can do the other way and it follows that
  \begin{align*}
   F - 2\mathcal{B} \leq \mathcal{F} \leq F + 2\mathcal{B}.
  \end{align*}

  Therefore, a flow is tamed by $\nu$ if and only if
  $\mathcal{F}(\nu, \cdot, \cdot)$ is uniformly bounded on $Y \times [0, \infty)$.
  However, the calculation under the metric $l^{\nu}$ is easier in many cases.
  \footnote{The calculation under the metric $l^{\nu}$ was first suggested to the author by Tian.}
  Some estimates in~\cite{CW4} can be improved.

 \begin{lemma}
   There is a uniform constant $A=A(\mathcal{B}, C_S, n)$ such that
 \begin{align}
   &\snorm{S}{l^{\nu}} < A\nu^{\frac{n}{2}}, \label{eqn: slinf}\\
   &\snorm{\nabla S}{l^{\nu}} < A \nu^{\frac{n+1}{2}}, \label{eqn: naslinf}
 \end{align}
 whenever $S \in H^0(Y, K_Y^{-\nu})$ is a unit norm section  (under the metric $l^{\nu}$).
 \label{lemma: boundl}
 \end{lemma}

 \begin{proof}
 For simplicity, we omit subindex $l^{\nu}$ in the proof.
 Note $\triangle_{\omega_{\varphi}} \snorm{S}{}^2= \snorm{\nabla S}{}^2-
 n\nu\snorm{S}{}^2$,
 the proof of inequality (\ref{eqn: slinf}) follows directly the
 proof of Lemma 3.1 in~\cite{CW4}.  So we only
 prove inequality (\ref{eqn: naslinf}).

  Direct calculation shows that
 \begin{align}
   \triangle_{\omega_{\varphi}} |\nabla S|^2 &=
 \snorm{\nabla \nabla S}{}^2 - (n+2)\nu \snorm{\nabla S}{}^2 +
 n\nu^2 |S|^2 + R_{i\bar{j}} \bar{S}_{,\bar{i}}S_{,j} \notag\\
 &=\snorm{\nabla \nabla S}{}^2 - [(n+2)\nu -1]\snorm{\nabla S}{}^2 +
 n\nu^2 |S|^2 -\dot{\varphi}_{, i\bar{j}} \bar{S}_{,\bar{i}}S_{,j}.
 \label{eqn: naseqn}
 \end{align}
 Note that  $S_{,i\bar{j}}= -\nu S g_{i\bar{j}}$,
 integration under measure $\omega_{\varphi}^n$ implies
 \begin{align*}
     \int_Y |\nabla \nabla S|^2  &= -n\nu^2
     +[(n+2)\nu -1]\int_Y \snorm{\nabla S}{}^2+\int_Y \dot{\varphi}_{,i\bar{j}}
     \bar{S}_{,\bar{i}}S_{,j}\\
 &=n\nu[(n+1)\nu-1]  -\int_Y \dot{\varphi}_i \bar{S}_{,\bar{i}\bar{j}}S_{,j}
    + n\nu \int_Y \dot{\varphi}_i\bar{S}_{,\bar{i}} S
 \end{align*}
 In view of $|\dot{\varphi}| \leq \mathcal{B}$, H\"older inequality
 implies
 \begin{align*}
  \int_Y \snorm{\nabla \nabla S}{}^2
 &\leq \mathcal{B} \left(\{\int_Y \snorm{\nabla \nabla S}{}^2\}^{\frac12}
    + n\nu \{\int_Y \snorm{S}{}^2 \}^{\frac12}\right)
   \{\int_Y \snorm{\nabla S}{}^2\}^{\frac12} + n\nu[(n+1)\nu-1]\\
 &=\sqrt{n\nu} \mathcal{B} \left(\{\int_Y \snorm{\nabla \nabla S}{}^2\}^{\frac12}
    + n\nu \right) + n\nu[(n+1)\nu-1]\\
 &\leq \frac12 \int_Y \snorm{\nabla \nabla S}{}^2
  + \frac12 n\nu \mathcal{B}^2 + (n\nu)^{\frac32} \mathcal{B} + n\nu[(n+1)\nu-1].
 \end{align*}
 It follows that
 \begin{align*}
    \int_Y \snorm{\nabla \nabla S}{}^2 \leq C \nu^2.
 \end{align*}
 for some constant $C=C(n,\mathcal{B})$.
 Combinging with the fact $\int_Y |\bar{\nabla}\nabla S|^2= n\nu^2$,
 Sobolev inequality implies
 \begin{align}
   \left(\int_Y \snorm{\nabla
   S}{}^{\frac{2n}{n-1}}\right)^{\frac{n-1}{n}}
   \leq C\nu^2.
 \label{eqn: nasn}
 \end{align}

  Fix $\beta>1$, multiplying $-\snorm{\nabla S}{}^{2(\beta-1)}$ to
  both sides of equation (\ref{eqn: naseqn}), we have
  \begin{align*}
   &\quad  \frac{4(\beta-1)}{\beta^2} \int_Y \left|  \nabla \snorm{\nabla S}{}^{\beta}\right|^2\\
   &=  -\int_Y (n\nu^2\snorm{S}{}^2 +\snorm{\nabla \nabla S}{}^2)\snorm{\nabla
      S}{}^{2(\beta-1)} + [(n+2)\nu-1] \int_Y \snorm{\nabla S}{}^{2\beta}\\
   & \qquad +\int_Y \dot{\varphi}_{,i\bar{j}}
   \bar{S}_{,\bar{i}}S_{,j}\snorm{\nabla S}{}^{2(\beta-1)}
  \end{align*}
  Note that
  \begin{align*}
  &\qquad \int_Y \dot{\varphi}_{,i\bar{j}}
   \bar{S}_{,\bar{i}}S_{,j}\snorm{\nabla S}{}^{2(\beta-1)}\\
  &=-\int_Y \dot{\varphi}_{,i} (\bar{S}_{,\bar{i}\bar{j}}S_{,j}
     + \bar{S}_{,\bar{i}} S_{,j\bar{j}})\snorm{\nabla S}{}^{2(\beta-1)}
  - (\beta-1) \int_Y \dot{\varphi}_{,i} \bar{S}_{,\bar{i}}S_{,j}(S_{k\bar{j}}\bar{S}_{,\bar{k}}
     + S_k\bar{S}_{\bar{k}\bar{j}}) \snorm{\nabla
     S}{}^{2(\beta-2)}\\
  &\leq \nu[\beta-1+n] \int_Y \dot{\varphi}_{,i} S
  \bar{S}_{,\bar{i}} \snorm{\nabla S}{}^{2(\beta-1)}
    + \mathcal{B} \beta \int_Y \snorm{\nabla \nabla S}{} \snorm{\nabla
    S}{}^{2\beta-1}\\
  \end{align*}
  H\"older inequality and Schwartz inequality yield that
  \begin{align*}
  &\mathcal{B}\nu[\beta-1+n] \{\int_Y
   \snorm{S}{}^2 \snorm{\nabla
  S}{}^{2(\beta-1)}\}^{\frac12}\{\int_Y \snorm{\nabla
  S}{}^{2\beta}\}^{\frac12}\\
  &\qquad+ \mathcal{B}\beta \{\int_Y \snorm{\nabla \nabla S}{}^2
   \snorm{\nabla S}{}^{2(\beta-1)}\}^{\frac12} \{\int_Y \snorm{\nabla
   S}{}^{2\beta}\}^{\frac12}\\
  &\leq n\nu^2 \left(\{\int_Y \snorm{S}{}^2 \snorm{\nabla S}{}^{2(\beta-1)}\}
     + \{\int_Y \snorm{\nabla \nabla S}{}^2 \snorm{\nabla S}{}^{2(\beta-1)}\}
   \right)\\
  &\qquad  +\{\frac{\mathcal{B}^2(\beta-1+n)^2}{4n} + \frac{\mathcal{B}^2 \beta^2}{4n\nu^2}\}
  \int_Y \snorm{\nabla  S}{}^{2\beta}.
  \end{align*}
  If $\beta \geq \frac{n}{n-1}$, combining previous three inequalities implies
  \begin{align*}
     \int_Y \left|  \nabla \snorm{\nabla S}{}^{\beta}\right|^2
   \leq C \beta(\beta^2 + \nu) \int_Y \snorm{\nabla S}{}^{2\beta}.
  \end{align*}
  In light of Sobolev inequality, we have
  \begin{align*}
   \left(\int_Y \snorm{\nabla S}{}^{\beta \cdot
   \frac{2n}{n-1}}\right)^{\frac{n-1}{n}}
  \leq C_S\{ \int_Y \left|  \nabla \snorm{\nabla S}{}^{\beta}\right|^2 + \int_Y \snorm{\nabla S}{}^{2\beta}\}
  \leq C \beta (\beta^2 + \nu) \int_Y \snorm{\nabla S}{}^{2\beta}.
  \end{align*}
  Let $k_0$ be the number such that $\lambda^{2k_0} \geq \nu > \lambda^{2(k_0-1)}$ where
 $\lambda=\frac{n}{n-1}$, we have
  \begin{align*}
   \left(\int_Y \snorm{\nabla S}{}^{\beta \cdot
   \frac{2n}{n-1}}\right)^{\frac{n-1}{n}}
   \leq
   \left\{
   \begin{array}{ll}
    C \beta^3 & \textrm{if} \quad \beta >\lambda^{k_0},\\
   (C \nu) \beta  &\textrm{if} \quad \beta \leq \lambda^{k_0}.\\
   \end{array}
   \right.
  \end{align*}
  Iteration implies
  \begin{align*}
  \left\{
  \begin{array}{l}
   \norm{\snorm{\nabla S}{}^2}{L^{\infty}}
    \leq C^{\sum_{k=1}^\infty \lambda^{-k}} \lambda^{\sum_{k=1}^\infty k\lambda^{-k}}
    \norm{\snorm{\nabla S}{}^2}{L^{\lambda^{k_0}}},\\
   \norm{\snorm{\nabla S}{}^2}{L^{\lambda^{k_0}}}
   \leq (C\nu)^{\sum_{k=1}^{k_0} \lambda^{-k}}  \norm{\snorm{\nabla
   S}{}^2}{L^{\lambda}}.
  \end{array}
  \right.
  \end{align*}
 Since $\sum_{k=1}^{k_0} \lambda^{-k} <\sum_{k=1}^{\infty}
 \lambda^{-k}=n-1$, combining these inequalities with inequality (\ref{eqn: nasn})
 gives us
 \begin{align*}
  \norm{\snorm{\nabla S}{}^2}{L^{\infty}} \leq C\nu^{n+1}.
 \end{align*}
 This proves inequality (\ref{eqn: naslinf}).
 \end{proof}

 Similarly, by sharpening the constants in Lemma 3.2 of~\cite{CW4},
 we obtain
 \begin{lemma}
   There is a uniform constant $A=A(\mathcal{B}, C_S, n)$ such that
 \begin{align}
   &\snorm{S}{h^{\nu}} < A\nu^{\frac{n}{2}}, \label{eqn: slinfh}\\
   &\snorm{\nabla S}{h^{\nu}} < A \nu^{\frac{n+1}{2}}, \label{eqn: naslinfh}
 \end{align}
 whenever $S \in H^0(Y, K_Y^{-\nu})$ is a unit norm section  (under the metric $h^{\nu}$).
 \label{lemma: boundh}
 \end{lemma}

 Lemma~\ref{lemma: boundl} and Lemma~\ref{lemma: boundh} clearly
 implies the following estimates.
 \begin{corollary}
 There is a uniform constant $A=A(\mathcal{B}, C_S, n)$ such that
 \begin{align*}
  &\max\{\mathcal{F}, F\} \leq \frac{\log A + n \log \nu}{\nu}, \\
  &\max\{\mathcal{G}, G\} \leq  A\nu^{n+1}.
 \end{align*}
 \end{corollary}

 \begin{proposition}
  Along the flow, $\mathcal{F}$ satisfies
  \begin{align*}
  \left\{
  \begin{array}{ll}
   \D{}{t} \mathcal{F}&=-\dot{\varphi}
     + \int_Y (\dot{\varphi} - \frac{\triangle
     \dot{\varphi}}{\nu})e^{\nu \mathcal{F}} \omega_{\varphi}^n,\\
   \triangle \mathcal{F}&= -n+ (\frac{1}{\nu}e^{-\nu \mathcal{F}} \mathcal{G} -\nu \snorm{\nabla
   \mathcal{F}}{}^2) \geq -n,\\
   \square \mathcal{F}&= n-\dot{\varphi} + \int_Y (\dot{\varphi} - \frac{\triangle
    \dot{\varphi}}{\nu}) e^{\nu \mathcal{F}} \omega_{\varphi}^n +
    (\nu \snorm{\nabla \mathcal{F}}{}^2 - \frac{1}{\nu}e^{-\nu \mathcal{F}} \mathcal{G}).
  \end{array}
  \right.
  \end{align*}
  $F$ satisfies
  \begin{align*}
  \left\{
  \begin{array}{ll}
   \D{}{t} F &= (n-R) - (1+\frac{1}{\nu})\int_Y
   (n-R) e^{\nu F} \omega_{\varphi}^n,\\
   \triangle F &= -R + (\frac{1}{\nu}e^{-\nu F}G - \nu \snorm{\nabla F}{}^2) \geq
   -R,\\
   \square F &= n-(1+\frac{1}{\nu})\int_Y (n-R) e^{\nu F} \omega_{\varphi}^n
     + (\nu \snorm{\nabla F}{}^2 - \frac{1}{\nu}e^{-\nu F}G).
  \end{array}
  \right.
  \end{align*}
 \end{proposition}

 \begin{proof}
 At $t= t_0$, suppose $\{S_{\beta}\}_{\beta=0}^{N_{\nu}}$ are
 orthonormal holomorphic sections of $H^0(Y, K_Y^{-1})$ under the
 metric $l^{\nu}(t_0)$. Assume $\{a_{\alpha \beta}(t)S_{\beta}\}_{\alpha=0}^{N_{\nu}}$
 are orthonormal holomorphic sections at time $t$ under the metric
 $l^{\nu}(t)$.  By these uniformization condition,  we have
 \begin{align*}
  &a_{\alpha \beta}(t_0) = \delta_{\alpha \beta},\\
  &\delta_{\alpha \gamma}=a_{\alpha \beta}\bar{a}_{\gamma \xi}
   \int_Y \langle  S_{\beta}, S_{\xi}\rangle  \omega_{\varphi}^n,\\
  &0=\dot{a}_{\alpha \beta}a_{\gamma \xi}
   \int_Y \langle  S_{\beta}, S_{\xi}\rangle  \omega_{\varphi}^n
  +a_{\alpha \beta} \dot{\bar{a}}_{\gamma \xi}
   \int_Y \langle  S_{\beta}, S_{\xi}\rangle  \omega_{\varphi}^n\\
  &\qquad +a_{\alpha \beta} \bar{a}_{\gamma \xi}
   \int_Y (-\nu \dot{\varphi}
  + \triangle \dot{\varphi}) \langle S_{\beta}, S_{\xi}\rangle
  \omega_{\varphi}^n.
 \end{align*}
 In particular, at $t=t_0$, using sum convention we have
 \begin{align*}
  &0=\dot{a}_{\alpha \gamma} + \dot{\bar{a}}_{\gamma \alpha} +
    \int_Y (-\nu \dot{\varphi} + \triangle \dot{\varphi})
    \langle S_{\alpha}, S_{\gamma} \rangle \omega_{\varphi}^n, \\
  &\left. \D{}{t} e^{\nu\mathcal{F}} \right|_{t=t_0}
    = \left.\D{}{t}
    \left(a_{\alpha \beta} \bar{a}_{\alpha \gamma}  \langle S_{\beta},
    S_{\gamma}\rangle \right) \right|_{t=t_0}\\
   &\qquad =\dot{a}_{\alpha \beta} \langle S_{\beta}, S_{\alpha}\rangle
   + \dot{\bar{a}}_{\alpha \gamma} \langle S_{\alpha}, S_{\gamma}\rangle
   +(-\nu \dot{\varphi}) \langle S_{\alpha}, S_{\alpha} \rangle.
 \end{align*}
 Fix $x \in Y$, at $t=t_0$, there is a unit norm section
 $S$ such that  $\snorm{S}{l^{\nu}}^2(x)= e^{\nu\mathcal{F}}$. Let
 $S_{0}=S$, then we have
 \begin{align*}
  \left. \D{}{t} e^{\nu\mathcal{F}} \right|_{t=t_0}
 = e^{\nu \mathcal{F}} (\dot{a}_{00} + \dot{\bar{a}}_{00} - \nu \dot{\varphi})
 =e^{\nu \mathcal{F}}
 (\int_Y (\nu \dot{\varphi} - \triangle \dot{\varphi}) e^{\nu \mathcal{F}} \omega_{\varphi}^n - \nu\dot{\varphi})
 \end{align*}
 On the other hand,
 \begin{align*}
    \triangle e^{\nu \mathcal{F}}= \langle \nabla S_{\alpha}, \nabla
    S_{\alpha}\rangle - n\nu \langle S_{\alpha}, S_{\alpha} \rangle
    = \mathcal{G} -n\nu e^{\nu \mathcal{F}}.
 \end{align*}
 It follows that
 \begin{align*}
  (\D{}{t} - \triangle) e^{\nu \mathcal{F}}=
   e^{\nu \mathcal{F}}
 \{\int_Y (\nu \dot{\varphi} - \triangle \dot{\varphi}) e^{\nu \mathcal{F}} \omega_{\varphi}^n + \nu
 (n-\dot{\varphi})\} - \mathcal{G}.
 \end{align*}

 Similarly, we can have
 \begin{align*}
  \D{}{t}e^{\nu F}&=e^{\nu F} \{\nu \triangle \dot{\varphi}
   - (\nu+1) \int_Y \triangle \dot{\varphi}e^{\nu F}
   \omega_{\varphi}^n\}\\
  &=e^{\nu F} \{\nu (n-R)
   - (\nu+1) \int_Y (n-R) e^{\nu F} \omega_{\varphi}^n\},\\
  \triangle e^{\nu F}&= G - \nu R e^{\nu F},\\
  \square e^{\nu F} &=e^{\nu F} \{ n\nu
   -(\nu+1)\int_Y (n-R)e^{\nu F}
   \omega_{\varphi}^n \}-G.
 \end{align*}

 From the evolution equation of $e^{\nu \mathcal{F}}$
 and $e^{\nu F}$, we can easily obtain the evolution equation of $\mathcal{F}$
 and $F$.
 \end{proof}

 \begin{remark}
  The advantage of $F$ appears when the evolution equation is
  calculated.  Every term in $\D{F}{t}$ is a geometric quantity.
  Suppose that $\int_0^{\infty} \int_Y (R-n)_{-} \omega_{\varphi}^n dt<
   \infty$ and $\int_0^{\infty} (R_{\max}(t)-n) dt< \infty$, then $F$ must be
   bounded from below and the flow is tamed.

   When we consider the convergence of metric space, the smooth
   convergence of $g_{i\bar{j}}$ will automatically induce the smooth
   convergence of $h^{\nu}= \det(g_{i\bar{j}})^{\nu}$.   Therefore,
   we prefer to use $h^{\nu}$ as the more natural metric of $K_Y^{-\nu}$
   under the \KRf.
 \end{remark}

   Since H\"ormarnder's estimate hold in the orbifold case. The
   bound in Lemma~\ref{lemma: boundh} implies the convergence of
   plurianticanonical sections when the underlying orbifolds
   converge.

 \begin{proposition}
  Suppose $Y$ is a Fano orbifold,  $\{(Y, g(t)), 0 \leq t < \infty\}$ is a \KRf
  without volume concentration.  Let $t_i$ be a sequence of time such that
   $\displaystyle  (Y, g(t_i)) \sconv (\hat{Y}, \hat{g})$
   for some Q-Fano normal variety $(\hat{Y}, \hat{g})$.
  Then for any fixed positive integer $\nu$ (appropriate for both $Y$ and $\hat{Y}$),
  the following properties hold.
  \begin{enumerate}
   \item  If $S_i \in H^0(Y, K_{Y}^{-\nu})$ and $\int_{Y} \snorm{S_i}{h^{\nu}(t_i)}^2
   \omega_{\varphi(t_i)}^n=1$, then by taking subsequence if necessary, we have
   $\hat{S} \in H^0(\hat{Y}, K_{\hat{Y}}^{-\nu})$ such that
   \begin{align*}
      S_i \sconv \hat{S},
      \quad \int_{\hat{Y}} \snorm{\hat{S}}{\hat{h}^{\nu}}^2 \hat{\omega}^n=1.
   \end{align*}

   \item  If $\hat{S} \in H^0(\hat{Y}, K_{\hat{Y}}^{-\nu})$ and
    $\int_{\hat{Y}} \snorm{\hat{S}}{\hat{h}^{\nu}}^2  \hat{\omega}^n =1$, then there is a subsequence  of
   sections $S_i \in H^0(Y_i, K_{Y_i}^{-\nu})$ satisfying
   \begin{align*}
   \int_{Y_i} \snorm{S_i}{h^{\nu}(t_i)}^2 \omega_{\varphi(t_i)}^n=1,
   \quad  S_i \sconv \hat{S}.
   \end{align*}

  \end{enumerate}
  \label{proposition: bundleconv}
  \end{proposition}

  Using this property, we can justify the tamedness condition by
  weak compactness exactly as Theorem 3.2 of~\cite{CW4}.

  \begin{theorem}
   Suppose $Y$ is a Fano orbifold,  $\{(Y, g(t)), 0 \leq t < \infty\}$ is a \KRf
   without volume concentration. Suppose this flow satisfies weak compactness, i.e., for every sequence $t_i \to \infty$, by
   passing to subsequence, we have
   \begin{align*}
      (Y, g(t_i)) \sconv (\hat{Y}, \hat{g}),
   \end{align*}
   where $(\hat{Y}, \hat{g})$ is a Q-Fano normal variety.

   Then this flow is tamed by a big constant $\nu$.
 \label{theorem: justtamed}
 \end{theorem}

 As mentioned in the introduction. Suppose $Y$ is
 an orbifold Fano surface, $\{(Y, g(t)), 0 \leq t <\infty\}$ is a \KRf solution. Then
 this flow has no volume concentration and satisfies weak
 compactness theorem.  Under the help of Perelman's functional,
 every weak limit $(\hat{Y}, \hat{g})$ must satisfy K\"ahler Ricci
 soliton equation on its smooth part.  On the other hand, the
 soliton potential function has uniform $C^1$-norm bound since it is
 the smooth limit of $-\dot{\varphi}(t_i)$. Therefore Uhlenbeck's
 removing singularity method applies and we obtain $(\hat{Y}, \hat{g})$
 is a smooth orbifold which can be embedded into $\CP^{N_{\nu}}$ by
 line bundle $K_{\hat{Y}}^{-\nu}$ for some big $\nu$(c.f.~\cite{Baily}).
 Then the following Theorem from~\ref{theorem: justtamed} directly.

 \begin{theorem}
  Suppose $Y$ is an orbifold Fano surface, $\{(Y, g(t)), 0 \leq t <\infty\}$
 is a \KRf solution. Then there is a big constant $\nu$ such that
 this flow is tamed by $\nu$.
 \label{theorem: surfacetamed}
 \end{theorem}

 \subsection{Properties of Tamed Flow}
  Follow~\cite{Tian91}, we define
  \begin{definition}
  Let $\mathscr{P}_{G, \nu, k}(Y, \omega)$ be the collection of all
  $G$-invariant functions of form
  $\displaystyle \frac{1}{2\nu}\log (\sum_{\beta=0}^{k-1} \norm{\tilde{S}_{\nu,
  \beta}}{h^{\nu}}^2)$, where $\tilde{S}_{\nu, \beta} \in H^0(K_Y^{-\nu})$ satisfies
  \begin{align*}
     \int_Y \langle \tilde{S}_{\alpha}, \tilde{S}_{\beta} \rangle_{h^{\nu}}
     \omega^n=\delta_{\alpha \beta},   \quad 0 \leq \alpha, \beta
     \leq k-1 \leq \dim(K_Y^{-\nu}) -1;
     \qquad h= \det g_{\omega}.
  \end{align*}
  Define
  \begin{align*}
      \alpha_{G, \nu, k} \triangleq
      \sup\{ \alpha |  \sup_{\varphi \in \mathscr{P}_{G,\mu, k}} \int_Y e^{-2\alpha \varphi} \omega^n <
      \infty\}.
  \end{align*}
  If $G$ is trivial, we denote $\alpha_{\nu, k}$
  as $\alpha_{G, \nu, k}$, denote $\mathscr{P}(\nu, k)$ as $\mathscr{P}(G, \nu, k)$.
  \label{definition: nualpha}
 \end{definition}

 The next definition follows~\cite{DK}.
 \begin{definition}
  Let $Y$ be a complex orbifold and $f$ is a plurisubharmonic function and
  $f \in L^1(Y)$.
  For any compact set $K \subset Y$, define
  \begin{align*}
    \alpha_K(f)= \sup \{ c \geq 0: \; e^{-2cf} \textrm{is $L^1$ on a neighborhood
    of}\; K\},
  \end{align*}
 This $\alpha_K(f)$ is called the complex singularity exponent of $f$ on $K$.
 \label{definition: singexp}
 \end{definition}

 If $f \in \mathscr{P}(\nu, k)$ and $\alpha < \alpha_{\nu, k}$,
 we have $\int_Y e^{-2\alpha f}< \infty$ by definition. Since the
 set $\mathscr{P}(\nu, k)$ is compact in $L^1(Y)$-topology (actually in $C^{\infty}$ topology).
 By the semicontinuity property proved in~\cite{DK}, we see there is
 a uniform constant $C_{\alpha, \nu, k}$ such that
 \begin{align*}
   \int_Y e^{-2\alpha f}< C_{\alpha, \nu, k},
   \quad
   \forall \; f \in \mathscr{P}(\nu, k).
 \end{align*}

 Suppose the flow is tamed by $\nu$. By rotating basis, we can
 choose $\{S_{\nu, \beta}^t\}_{\beta=0}^{N_{\nu}}$ and
 $\{\tilde{S}_{\nu, \beta}^t\}_{\beta=0}^{N_{\nu}}$ as orthonormal
 basis of $H^0(K_Y^{-\nu})$ under the metric $h^{\nu}(t)$ and $h^{\nu}(0)$
 respectively, and they satisfies
 \begin{align*}
     S_{\nu, \beta}^t= a(t)\lambda_{\beta}(t)\tilde{S}_{\nu,
     \beta}^t, \quad
     0< \lambda_0(t) \leq \lambda_1(t) \leq \cdots \leq
     \lambda_{N_{\nu}}(t)=1.
 \end{align*}
 As in~\cite{CW4}, we have the partial $C^0$-estimate
 \begin{align*}
   \snorm{\varphi  - \sup_Y \varphi -\frac{1}{\nu} \log
   \sum_{\beta=0}^{N_{\nu}} \snorm{\lambda_{\beta}(t)\tilde{S}_{\nu, \beta}^t}{h_0^{\nu}}^2}{} <
   C.
 \end{align*}
 This yields
 \begin{align*}
   \int_Y e^{-\alpha(\varphi -\sup_Y \varphi)} \omega^n
    &< e^C \int_Y \left(\sum_{\beta=0}^{N_{\nu}}
     \snorm{\lambda_{\beta}(t)\tilde{S}_{\nu, \beta}^t}{h_0^{\nu}}^2
     \right)^{-\frac{\alpha}{\nu}} \omega^n\\
   &< e^C  \int_Y \left(\sum_{\beta=N_{\nu}-k+1}^{N_{\nu}}
     \snorm{\lambda_{\beta}(t)\tilde{S}_{\nu, \beta}^t}{h_0^{\nu}}^2
     \right)^{-\frac{\alpha}{\nu}} \omega^n\\
   &\leq e^C \lambda_{N_{\nu}-k+1}^{-\frac{2 \alpha}{\nu}}
    \int_Y \left(\sum_{\beta=N_{\nu}-k+1}^{N_{\nu}}
     \snorm{\tilde{S}_{\nu, \beta}^t}{h_0^{\nu}}^2
     \right)^{-\frac{\alpha}{\nu}} \omega^n\\
   &<e^C C_{\alpha, \nu, k} \lambda_{N_{\nu}-k+1}^{-\frac{2
   \alpha}{\nu}}.
 \end{align*}
 Plug in the equation $\dot{\varphi}= \log \frac{\omega_{\varphi}^n}{\omega^n} +
 \varphi+u_{\omega}$ and note that $\dot{\varphi}, u_{\omega}$ are bounded, we have
 \begin{align*}
   \int_Y e^{(1-\alpha)\varphi + \alpha \sup_Y \varphi} \omega_{\varphi}^n <
   C'(\alpha, \nu, k) \lambda_{N_{\nu}-k+1}^{-\frac{2\alpha}{\nu}}
 \end{align*}
 The convexity of exponential function implies
 \begin{align}
  (1-\alpha) \frac{1}{V} \int_Y \varphi \omega_{\varphi}^n + \alpha \sup_Y \varphi < C''(\alpha, \nu, k) -
  \frac{2\alpha}{\nu} \log \lambda_{N_{\nu}-k+1}.
 \label{eqn: intsupsum}
 \end{align}
 whenever $\alpha< \alpha_{\nu, k}$. Using this estimate, we can
 obtain the following two convergence theorems as in~\cite{CW4}.

 \begin{theorem}
  Suppose $\{(Y^n, g(t)), 0 \leq t < \infty \}$ is a \KRf tamed by $\nu$.
  If $\alpha_{\nu, 1}> \frac{n}{(n+1)}$,
  then $\varphi$ is uniformly bounded along this flow. In
  particular, this flow converges to a KE metric exponentially fast.
 \label{theorem: nuconv}
 \end{theorem}
 \begin{proof}
 Choose $\alpha \in (\frac{n}{n+1}, \alpha_{\nu, 1})$. Put $k=1$ into inequality (\ref{eqn: intsupsum}),
 we have
 \begin{align*}
  (1-\alpha) \frac{1}{V} \int_Y \varphi \omega_{\varphi}^n + \alpha \sup_Y \varphi
  < C(\alpha, \nu).
 \end{align*}
 Together with
  $\frac{1}{V} \int_Y (-\varphi) \omega_{\varphi}^n \leq n \sup_Y
  \varphi+C$, it implies
 \begin{align*}
  \{\alpha - n(1-\alpha)\}\sup_{Y} \varphi < C.
 \end{align*}
 As $\alpha>\alpha_{\nu, 1}>\frac{n}{n+1}$, we have $\alpha - n(1-\alpha)>0$,
 this yields that $\displaystyle \sup_Y \varphi$ is
 uniformly bounded from above. Therefore, $\varphi$ is uniformly
 bounded.
 \end{proof}

 \begin{theorem}
 Suppose $\{(Y^n, g(t)), 0 \leq t < \infty \}$ is a \KRf tamed by $\nu$.
 If $\alpha_{\nu, 2}>\frac{n}{n+1}$ and
 $\alpha_{\nu, 1} > \frac{1}{2- \frac{n-1}{(n+1) \alpha_{\nu,2}}}$, then $\varphi$ is
 uniformly bounded along this flow.  In particular, this flow converges to a KE
 metric exponentially fast.
 \label{theorem: nuconvr}
 \end{theorem}

 \begin{proof}
   We argue by contradiction. Suppose that $\varphi$ is not
   uniformly bounded.

    Then there must be a sequence of $t_i$ such that
    $\displaystyle \sup_{Y} \varphi(t_i) \to  \infty$.
    We claim that $\lambda_{N_{\nu}-1}(t_i) \to 0$. Otherwise, $\log \lambda_{N_{\nu}-1}(t_i)$ is uniformly
   bounded. Choose $\alpha \in (\frac{n}{n+1}, \alpha_{\nu, 2})$.
   Combining  $\frac{1}{V} \int_Y (-\varphi) \omega_{\varphi}^n \leq n \sup_Y
  \varphi+C$ and the inequality (\ref{eqn: intsupsum})
  in the case $k=1$,  the same argument as in the proof of Theorem~\ref{theorem: nuconv}
  implies that $\displaystyle \sup_Y \varphi(t_i)$ is uniformly
  bounded. This contradicts to our assumption!

  Note that $\R$-coefficient \Poincare duality holds on orbifold, singularities on $Y$
  are isolated which can be included in small geodesic balls with few contribution to integration.
  Since $\lambda_{N_{\nu}-1}(t_i) \to 0$, as in~\cite{Tian91}, for every small $\delta>0$,
  we have
  \begin{align*}
 \frac{1}{V} \int_Y \st \partial X_{t_i} \wedge
     \bar{\partial} X_{t_i} \wedge \omega^{n-1}
 \geq -\frac{(1-\delta)}{\nu} \log \lambda_{N_{\nu}-1}(t_i) -C
 \end{align*}
 for large $i$.
 Here
  $\displaystyle X_{t_i}=\frac{1}{\nu} \log \sum_{\beta=0}^{N_{\nu}} \snorm{\lambda_{\beta}(t_i)\tilde{S}_{\nu,
   \beta}^{t_i}}{h_0^{\nu}}^2$.
  For notation simplicity, we omit the subindex $t_i$ from now
  on. It follows that
  \begin{align*}
   &\quad \sum_{j=0}^{n-1}
     \frac{j}{V} \int_Y \st \partial \varphi \wedge
     \bar{\partial} \varphi \wedge \omega^j \wedge
     \omega_{\varphi}^{n-1-j}\\
 &\geq \frac{n-1}{V} \int_Y \st \partial \varphi \wedge
     \bar{\partial} \varphi \wedge \omega^{n-1}\\
 &\geq \frac{n-1}{V} \int_Y \st \partial X \wedge
     \bar{\partial} X \wedge \omega^{n-1} -C\\
 &\geq -(1-\delta) \cdot \frac{(n-1)}{\nu}
     \cdot \log \lambda_{N_{\nu}-1} -C.
 \end{align*}
 Plug this into inequality (\ref{eqn: intsupsum}) in the case $k=1$,
 we arrive

 \begin{align*}
   (1-\alpha) \frac{1}{V} \int_Y \varphi \omega_{\varphi}^n + \alpha \sup_Y \varphi < C(\alpha, \nu)
   +\frac{1}{1-\delta} \cdot \frac{2\alpha}{(n-1)}\sum_{i=0}^{n-1}
     \frac{i}{V} \int_Y \st \partial \varphi \wedge
     \bar{\partial} \varphi \wedge \omega^i \wedge
     \omega_{\varphi}^{n-1-i}.
 \end{align*}
 Combining it with
 \begin{align}
   \frac{1}{V} \int_Y (-\varphi) \omega_{\varphi}^n \leq
  n \sup_Y \varphi  - \sum_{i=0}^{n-1}
     \frac{i}{V} \int_Y \st \partial \varphi \wedge
     \bar{\partial} \varphi \wedge \omega^i \wedge
     \omega_{\varphi}^{n-1-i} + C
 \end{align}
 we have
 \begin{align*}
     \left(2A\alpha -(1-\alpha)\right) \frac{1}{V} \int_Y
     (-\varphi) \omega_{\varphi}^n < \alpha (2An -1)  \sup_Y \varphi +C.
 \end{align*}
 where $A=\frac{1}{(n-1)(1-\delta)}$.
 Combining this with the estimate (\ref{eqn: intsupsum}) for $k=1$
 implies
  \begin{align*}
     \left(2A\alpha -(1-\alpha)\right) \frac{1}{V} \int_Y
     (-\varphi) \omega_{\varphi}^n < (2A n-1) \alpha  \sup_Y \varphi +C
  < (2An -1) \frac{\alpha}{\beta}(1-\beta) \frac{1}{V} \int_Y(-\varphi) \omega_{\varphi}^n
    +C.
 \end{align*}
 where $\beta$ is any number less that $\alpha_{\nu, 1}$, $A=\frac{1}{(n-1)(1-\delta)}$. Therefore,
 we have
 \begin{align}
   \{ (2A +1 - \frac{1-\beta}{\beta} (2nA -1))\alpha -1\} \frac{1}{V} \int_Y(-\varphi)
   \omega_{\varphi}^n<C.
 \label{eqn: intphi}
 \end{align}

 If  $\alpha_{\nu, 1} > \frac{1}{2-\frac{(n+1)}{(n-1)\alpha_{\nu, 2}}}$, then
 we can find $\beta$  a little bit less than $\alpha_{\nu, 1}$,
 $\alpha$ a little bit less than $\alpha_{\nu, 2}$, $A$ a little bit
 greater than $\frac{1}{n-1}$ such that
 \begin{align*}
  (2A +1 - \frac{1-\beta}{\beta} (2nA -1))\alpha -1>0.
 \end{align*}
 Recall our subindex $t_i$ in inequality (\ref{eqn: intphi}),
 we have $\frac{1}{V} \int_Y(-\varphi_{t_i}) \omega_{\varphi_{t_i}}^n$
 is  uniformly bounded from above. This implies that $\displaystyle \sup_Y \varphi_{t_i}$
 is uniformly bounded. Contradiction!
 \end{proof}

 \section{Some Applications and Examples}
 The following theorem is a direct corollary of
 Theorem~\ref{theorem: surfacetamed},
 Theorem~\ref{theorem: nuconv} and Theorem~\ref{theorem: nuconvr}.

 \begin{theorem}
  Suppose that $Y$ is an orbifold Fano surface such that one of the
  following two conditions holds for every large integer $\nu$,
  \begin{itemize}
  \item  $\alpha_{\nu, 1}> \frac23$.
  \item  $\alpha_{\nu, 2} > \frac23, \quad \alpha_{\nu, 1} > \frac{1}{2- \frac{1}{3\alpha_{\nu, 2}}}$.
  \end{itemize}
  Then $Y$ admits a KE metric.
  \label{theorem: KEexist}
 \end{theorem}

 There are a lot of orbifold Fano surfaces where Theorem~\ref{theorem: KEexist}
 can be applied.   For simplicity, we only consider the good case:
 every singularity is a rational double point. This kind of orbifolds are called
 Gorenstein log del Pezzo surfaces.

 Let's first recall some definitions.

 \begin{definition}
   Suppose that $X$ is a normal variety and $D=\sum d_i D_i$ is a
   $\Q$-cartier divisor on $X$ such that $K_X +D$ is $\Q$-cartier
   and let $f: Y \to X$ be a birational morphism, where $Y$ is
   normal. We can write
   \begin{align*}
    K_{Y} \sim_{\Q} f^* (K_X +D) + \sum a(X, D, E)E.
   \end{align*}
  The discrepancy of the log pair $(X, D)$ is the number
  \begin{align*}
     discrep(X, D)= \inf_E \{ a(X, D, E) | E \; \textrm{is exceptional divisor over} \;
     X\}.
  \end{align*}
 The total discrepancy of the log pair $(X, D)$ is the number
 \begin{align*}
   totaldiscrep(X, D)= \inf_E \{a(X,D,E) | E \; \textrm{is divisor over} \;
   X\}.
 \end{align*}
 We say that the log pair $K_X +D$ is
 \begin{itemize}
 \item  Kawamata log terminal (or log terminal) if and only if
   $totaldiscrep(X, D)>-1$.
 \item log canonical iff $discrep(X, D)\geq -1$.
 \end{itemize}
 \end{definition}

  Assume now that $X$ is a variety with log terminal singularities,
  let $Z \subset X$ be a closed subvariety and let $D$ be an
  effective $\Q$-Cartier divisor on $X$. Then the number
  \begin{align*}
    lct_Z(X, D) = \sup\{ \lambda \in \Q | \textrm{the log pair} \; (X, \lambda D)\; \textrm{is log canonical along}
     \; Z\}.
  \end{align*}
  Let $x$ be a point in $X$, $f$ be a local defining holomorphic
  function of divisor $D$ around $x$, then we have
  \begin{align*}
    lct_x(X, D)= \alpha_x(\log f)
  \end{align*}
  where $\alpha_x(\log f)$ is the singularity exponent of plurisubharmonic function $\log f$ around point
  $x$. (c.f. definition~\ref{definition: singexp}).

 \begin{definition}
 \begin{align*}
  lct_{\nu}(X)= \inf\{ lct(X, \frac{1}{\nu}D) | D \; \textrm{effective} \; \textrm{$\Q$-divisor on}
  \; X \;  \textrm{such that} \; D \in |-\nu K_X|\}.
 \end{align*}
 The global log canonical threshold of $X$ is the number
 \begin{align*}
    lct(X)= \inf\{lct(X, D)| D \; \textrm{effective divisor of} \; X \;\textrm{such that}
        \; D \sim_{\Q} -K_X\}.
 \end{align*}
 \end{definition}

 It's not hard to see $lct_{\nu}(X)= \alpha_{\nu, 1}$.  According to the proof
 of Demailly (c.f.~\cite{ChS},~\cite{SYl}), we know
 \begin{align*}
  \alpha(X) = lct(X) = \lim_{\nu \to \infty} lct_{\nu}(X) = \lim_{\nu \to
  \infty} \alpha_{\nu, 1}.
 \end{align*}
 Therefore, we have
 \begin{align*}
  \infty=\alpha_{\nu, N_{\nu}+1} \geq \cdots \geq \alpha_{\nu, 3} \geq
  \alpha_{\nu, 2} \geq \alpha_{\nu, 1}(X)=lct_{\nu}(X)  \geq lct(X) = \alpha(X).
 \end{align*}

 The calculation of $\alpha_{\nu, k}$ is itself a very interesting
 problem (c.f.~\cite{SYl},~\cite{ChS}).  Here we will
 use some results calculated in~\cite{Kosta}.

 \begin{lemma}[\cite{Kosta}]
    Let $Y$ be a Gorenstein log del Pezzo surface, every singularity of $Y$ is of type
 $\A_k$.  Suppose $Y$ satisfies one of the following conditions.
 \begin{itemize}
 \item $Y$ has only singularities of type $\A_1$ or $\A_2$
   and $K_Y^2=1$. $|-K_Y|$ has a cuspidal curve $C$ such that
   $Sing(C)$ contains an $\A_2$ singularity.
 \item $Y$ has one singularity of type $\A_5$ and $K_Y^2=1$.
 \item $Y$ has one singularity of type $\A_6$ and $K_Y^2=1$.
 \end{itemize}
 Then $\alpha_{\nu, 1} \geq \frac23$
    and $\alpha_{\nu, 2}> \frac23$.
 \label{lemma: cacalpha}
 \end{lemma}

 \begin{proof}
   The proof argues case by case and the main ingredients are
   contained in~\cite{Kosta} already.  For simplicity, we only give a
   sketch proof of the second case.

   If $f \in \mathscr{P}(\nu, 1)$ and $\alpha_x(f) \leq \frac23$,
   one can show that $f = \frac{1}{2\nu} \log \snorm{S}{h_0^{\nu}}^2$
   for some $S \in H^0(K_Y^{-\nu})$. Moreover, $x$ is the unique singularity
   of type $\A_5$ and $S= (S')^{\nu}$
   for some $S' \in H^0(K_Y^{-1})$.   $Z(S')$ is the unique
   divisor passing through $x$ such that
    $lct_x(Y, Z(S'))= \frac23$.

   For every $\varphi \in \mathscr{P}(\nu, 2)$,
   we have $e^{2\nu \varphi} = e^{2\nu \varphi_1} + e^{2\nu \varphi_2}$
   where
   \begin{align*}
    \varphi_1= \frac{1}{\nu} \log \snorm{S_1}{h_0^{\nu}}^2, \quad
    \varphi_2= \frac{1}{\nu} \log \snorm{S_2}{h_0^{\nu}}^2; \quad
    \int_Y \langle S_1, S_2\rangle_{h_0^{\nu}} \omega_0^n =0.
   \end{align*}
   Clearly,  for every point $y \in Y$, we have
   \begin{align*}
     \alpha_y(\varphi) \geq \max \{ \alpha_y(\varphi_1),
     \alpha_y(\varphi_2)\}> \frac23.
   \end{align*}
   Since $\alpha_y(\varphi_1), \alpha_y(\varphi_2)$ can only achieve finite possible
   values, we have
   \begin{align*}
      \inf_{y \in Y, \varphi \in \mathscr{P}(\nu, 2)}
      \alpha_y(\varphi) > \frac23.
   \end{align*}
   By the compactness of $Y$ and the semicontinuity property proved
   in~\cite{DK}, we have the inequality $\alpha_{\nu, 2} > \frac23$.
 \end{proof}

  Therefore, Theorem~\ref{theorem: KEexist} applies and we know KE metrics
  exist on such orbifolds $Y$ in Lemma~\ref{lemma: cacalpha}.
  Together with Theorem 1.6 of~\cite{Kosta} and Theorem 5.1 of~\cite{SYl}, we have proved the following theorem.

   \begin{theorem}
  Suppose $Y$ is a cubic surface with only one ordinary double point, or $Y$ is
  a degree $1$ del Pezzo surface having only Du Val singularities of type $\A_k$ for $k \leq 6$. Starting from any
   metric $\omega$ satisfying $[\omega]=2\pi c_1(Y)$,
   the \KRf will converge to a KE metric on $Y$.  In particular, $Y$ admits a KE metric.
  \end{theorem}

 \begin{remark}
  If we consider $\alpha_{G, \nu, k}$ instead of $\alpha_{\nu, k}$
  for some finite group $G \subset Aut(Y)$,  it's still possible to study
  the existence of KE metrics on degree $1$  Gorenstein log Del Pezzo surfaces
  with $\A_7$ or $\A_8$ singularities.
 \end{remark}


%
%
%

  \vspace{1in}

 Bing  Wang,  Department of Mathematics, Princeton University,
  Princeton, NJ 08544, USA; bingw@math.princeton.edu

  \end{document}